\documentclass[11pt]{amsart}
\usepackage{amsfonts}
\usepackage{amsmath}
\usepackage{graphicx}
\usepackage{epsfig}
\usepackage{graphicx}
\usepackage{amssymb}
\usepackage{mathrsfs}
\usepackage{amsthm}
\usepackage[all]{xy}
\usepackage{pdflscape}
\usepackage{subfigure}
\usepackage[hang,small,bf]{caption}
\usepackage{framed}
\usepackage{epstopdf}
\usepackage{amscd}
\usepackage{fullpage}

\newtheorem*{mthm}{Main Theorem}

\input amssym.def

\newcommand{\Char}{\text{Char}}

\newcommand{\s}{\mathfrak{s}}
\newcommand{\comment}[1]{}
\newcommand{\abs}[1]{\left\vert #1 \right\vert}
\newcommand{\half}{{\tfrac{1}{2}}}
\newcommand{\Spinc}{\text{Spin}^c}
\newcommand{\Spin}{\text{Spin}^c}

\newcommand{\Charm}{\text{Char}^*}
\DeclareMathOperator{\rank}{rank}
\DeclareMathOperator{\PD}{PD}

\DeclareMathOperator{\sgn}{sgn}
\DeclareMathOperator{\coker}{coker}
\DeclareMathOperator{\Hom}{Hom}
\DeclareMathOperator{\id}{id}
\theoremstyle{plain}
\newtheorem{thm}{Theorem}[section]
\newtheorem{lem}[thm]{Lemma}
\newtheorem{prop}[thm]{Proposition}
\newtheorem{cor}[thm]{Corollary}

\newtheorem{qn}[thm]{Question}
\newtheorem{conj}[thm]{Conjecture}

\title{Pretzel Knots with Unknotting Number One}

\author{Dorothy Buck, Julian Gibbons, Eric Staron}
\begin{document}

\begin{abstract}
We provide a partial classification of the 3-strand pretzel knots $K = P(p,q,r)$ with unknotting number one. Following the classification by Kobayashi and Scharlemann-Thompson for all parameters odd, we treat the remaining families with $r$ even. We discover that there are only four possible subfamilies which may satisfy $u(K) = 1$. These families are determined by the sum $p+q$ and their signature, and we resolve the problem in two of these cases. Ingredients in our proofs include Donaldson's diagonalisation theorem (as applied by Greene), Nakanishi's unknotting bounds from the Alexander module, and the correction terms introduced by Ozsv\'ath and Szab\'o. Based on our results and the fact that the 2-bridge knots with unknotting number one are already classified, we conjecture that the only 3-strand pretzel knots $P(p,q,r)$ with unknotting number one that are not 2-bridge knots are $P(3,-3,2)$ and its reflection.

\end{abstract}

\maketitle

\section{Introduction}

The unknotting number $u(K)$ is the minimal number of times a knot $K$ must be passed through itself in order to unknot it, an invariant that is at once easy to define yet at the same time almost always extremely difficult to compute. Indeed, it took many years to calculate $u(K)$ for the majority of knots with ten or fewer crossings, and while exhibiting an upper bound is straightforward (by performing an unknotting), lower bounds are more elusive: it is generally not known which knot diagrams will realise the actual unknotting number (see \cite{Bleiler}, \cite{nakanishi2}, and \cite{Stoimenow}).

One classical lower bound for the unknotting number is the knot signature, $\sigma(K)$, which satisfies $\abs{\sigma(K)} \leq 2u(K)$ (see \cite{Murasugi}). For example, if $u(K) = 1$, it follows that $\abs{\sigma(K)} = 0,2$. This condition is often the first port of call when investigating unknotting number. As one might expect, however, it is rarely sufficient $-$ infinite families of knots with the same signature but wildly different unknotting numbers are known to exist. It is only in certain cases, for example when $K$ is a torus knot, that the bound is tight (\cite{Mrowka} and \cite{Rasmussen}).

Specific to the case of unknotting number one, there are a number of other topological obstructions, many concerning the double branched cover $\Sigma(K)$. The most important of these for this paper is the Montesinos theorem: if $u(K) = 1$, then $\Sigma(K)$ arises as half-integral surgery on some knot $\kappa \subset S^3$. That is, $\Sigma(K) = S^3_{\pm D/2}(\kappa)$ (see \cite{montesinos}). This has various implications: cyclic $H_1(\Sigma(K))$, restrictions on the 4-manifolds with $\Sigma(K)$ as boundary, and symmetries in the correction terms of $\Sigma(K)$ (see \cite{OSAbsolute} and \cite{OSUnknot}).

Following both these leads, our main result in the present work is a partial classification of the 3-strand pretzels $K = P(p,q,r)$ with unknotting number one. Such knots are unchanged by permutations of their parameters, and have reflections given by
$$\overline{P(p,q,r)} = P(-p,-q,-r).$$
For $K$ to be a \emph{bona fide} knot, we require either that all three parameters be odd, or that exactly one of them be even (say $r = 2m$). The first of these cases (all odd) has been studied independently by Kobayashi \cite{Kobayashi} and Scharlemann and Thompson \cite{Scharlemann}, who give the criterion that
$$u(K)=1 \iff \pm\{1,1\}\text{ or }\pm\{3,-1\} \subset \{p,q,r\},$$
and thus our work concentrates on the case $P(p,q,2m)$. As a consequence of fact that $u(\overline{K}) = u(K)$, we assume that $2m$ is non-negative, and, having dealt with the case $m = 0$ early on, thereafter restrict our attention to $m > 0$.

As a final piece of set-up, recall that Kanenobu and Murakami \cite{Kanenobu} and Torisu \cite{Torisu} have given a complete description of the 2-bridge knots with unknotting number one. Since the double branched cover of a pretzel knot is Seifert fibred over $S^2$, it follows that $P(p,q,r)$ is not a 2-bridge knot if and only if all three of $p,q,r \neq \pm 1$ (or else the double branched cover would have fewer than three exceptional fibres and therefore be a lens space). As $r$ is even, $r \neq \pm 1$, and so our primary interest will be when $p,q \neq \pm 1$.

\subsection{Main Results}

Our first result, determined by way of knot signatures, says that there are only four families of 3-strand pretzel knots (excluding 2-bridge knots), $r$ even, which stand a chance of satisfying $u(K) = 1$. Having identified these families according to their values $p+q$, our main theorem is then the following.

\begin{mthm} Suppose that $K=P(p,q,2m)$, $m \neq 0$, is a pretzel knot with unknotting number one. Then, up to reflection, $p+q = 0,\pm2,4$ and $m > 0$. Moreover:
\begin{enumerate}
\item If $p+q=-2$, then $K = P(1,-3,2m), P(-1,-1,2m)$ (all 2-bridge);
\item If $p+q=0$, then $K = P(3,-3,2)$ (which is not 2-bridge).
\end{enumerate}
\end{mthm}

The table below indicates which pretzels in each family have unknotting number one, together with our conjectures. We present it as a more digestible version of the theorem's conclusions.

\begin{center}
\begin{tabular}{c||c|c}
Family &Knots must be... & Conjecture\\
\hline
\hline
$p+q=-2$ &$P(1,-3,2m),P(-1,-1,2m)$ &$-$\\
$p+q=0$ &$P(3,-3,2)$ & $-$\\
\hline
$p+q=2$ & unknown &$P(3,-1,2m),P(1,1,2m)$\\ 
$p+q=4$ & unknown &$P(3,1,2),P(5,-1,4),P(5,-1,2)$
\end{tabular}
\end{center}

Most of these are in fact 2-bridge as at least one parameter is $\pm1$. Hence, we have the following conjecture:

\begin{conj}
The only 3-strand pretzel knots $P(p,q,r)$ with unknotting number one that are not 2-bridge knots are $P(3,-3,2)$ and its reflection.
\end{conj}

The pretzel referred to in this conjecture is the following:

\begin{figure}[h]
\begin{center}
\includegraphics[trim = 4in 5in 4in 4.5in, scale = .4]{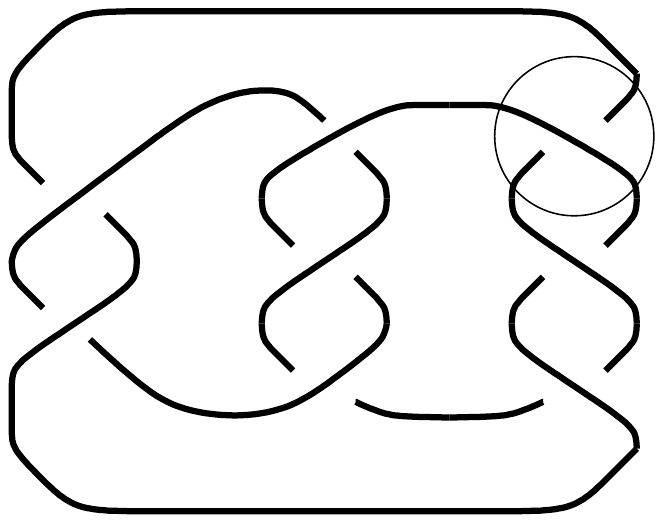}
\end{center}
\end{figure} 
\noindent and the circle indicates the unknotting crossing.

\subsection{Motivation}

This work was motivated by the following question: Which algebraic knots, in the sense of Conway, satisfy $u(K) = 1$? A complete treatment of algebraic knots can be found in \cite{gordon} and \cite{thistlethwaite}, but in brief, the distinct types are 2-bridge, large algebraic, and Montesinos length three, with the characterisation being split according to the topology of their double covers. To wit, we have the following division.\\

\begin{center}
\begin{tabular}{c||c|c|c}
$K$ &2-bridge &large algebraic &Montesinos length three\\
\hline
\hline
$\Sigma(K)$ &lens space &graph manifold &atoroidal Seifert fibred \\
 & &(toroidal) &($S^2$ with 3 exceptional fibres)\\
\end{tabular}\\
\end{center}

As stated previously, Kanenobu and Murakami have solved the problem for 2-bridge knots in \cite{Kanenobu}, and this solution was later generalised using Gordian distance by Torisu \cite{Torisu}. The large algebraic case is dealt with by Gordon and Luecke \cite{gordon} in terms of the constituent algebraic tangles of $K$. However, because the double branched cover of a Montesinos knot of length three is neither a lens space nor toroidal, neither of these results apply. It is then natural to ask the following question.

\begin{qn} \label{qn1} Which Montesinos knots of length three have unknotting number one? \end{qn}

In \cite{Torisu} Torisu makes the following conjecture. He proves the theorem immediately afterwards as evidence for his claim.

\begin{conj} [Torisu] \label{TorisuConj}
Let $K$ be a Montesinos knot of length three. Then $u(K)=1$ if and only if $K=\mathcal{M}(0;(p,r),(q,s),(2mn\pm1,2n^2))$, where $p$, $q$, $r$, $s$, $m$, and $n$ are non-zero integers, $m$ and $n$ are coprime, and $ps+rq=1$.
\end{conj}

\begin{thm} [Torisu]
Let $K$ be a Montesinos knot of length three and suppose the unknotting operation is realised in a standard diagram.  Then $u(K)=1$ if and only if it has the form in Conjecture \ref{TorisuConj}. 
\end{thm}

A proof of the following conjecture (see Conjecture 4.8 of \cite{GordonProblems}) would also prove Conjecture \ref{TorisuConj}.

\begin{conj}[Seifert fibering conjecture]\label{conj:SF} For a knot in $S^3$ which is neither a torus knot nor a cable of a torus knot, only integral surgery slopes can yield a Seifert fibred space.
\end{conj}

A complete explanation of why Conjecture \ref{conj:SF} implies Conjecture \ref{TorisuConj} can be found in \cite{Torisu}. In short, if a Montesinos knot $K\subset S^3$ has unknotting number one, then $\Sigma(K)$, a Seifert fibred space, equals $S^3_{\pm D/2}(\kappa)$, where $D$ is odd and $\kappa \subset S^3$ is a knot. If the Seifert fibering conjecture is true, then $\kappa$ is either a torus knot or a cable of a torus knot.  In either case, Dehn surgery on these knots is well understood (see Moser \cite{Moser}), and after some numerical calculations the desired result is achieved.

Of interest to us is what Torisu's conjecture predicts about 3-strand pretzel knots with unknotting number one. After a little work, it is not difficult to see that it not only suggests the results proved in this paper, but also implies our conjecture in the $p+q=2$ case. Thus, our work can be seen as a partial proof of Torisu's conjecture.

\subsection{Organisation}

As foreshadowed, we first use the knot signature to separate our knots into four types of candidates for $u(K) = 1$. These are split according to whether $p+q = 0,\pm2, 4$. All four require different approaches.

When $p+q = -2$, we use the Montesinos theorem coupled with a certain plumbing for $\Sigma(K)$ to glue together a closed, oriented, simply connected, smooth, negative-definite 4-manifold, and thence apply Donaldson's diagonalisation theorem. This turns out to be insufficient as an obstruction to unknotting number one, so to make more progress we use a strengthened version of this approach due to Greene. The result, in the case $p+q = -2$, is that $K$ must be 2-bridge to satisfy $u(K) = 1$. We conjecture that this is true in greater generality (i.e. for the remaining $p+q=2,4$ cases).

When $p+q = 0$, we do two things. First, we use the Alexander module of the pretzel to conclude that $m = 1$. Second, we employ the correction terms of $\Sigma(K)$ as defined by Ozsv\'ath and Szab\'o to prove that $p = 3$. This last part is a two-step procedure in which we first consider the Ozsv\'ath-Szab\'o obstruction modulo $\mathbb{Z}$ to narrow down possible $\Spinc$-structure labellings compatible with the required symmetries, before making use of the full obstruction to complete the proof in these restricted cases.

Our results give us evidence for the truth of our conjecture, which would leave only the chiral knot $P(3,-3,2)$ and its reflection as the non-2-bridge knots with unknotting number one.

\subsection{Acknowledgements}

The authors would like to extend their thanks to Cameron Gordon, Josh Greene, Raymond Lickorish, and Andrew Lobb for helpful discussions, and to Ana Lecuona and Brendan Owens for their careful reading of and insightful comments on preliminary versions. They would also like to thank their reviewers for many helpful suggestions. DB is supported in part by EPSRC Grants EP/H0313671, EP/G0395851 and EP/J1075308, and thanks the LMS for their Scheme 2 Grant. JG is supported by the Rector's Award, SOF, and Roth Fellowship at Imperial College London. ES is partially supported by NSF RTG Grant DMS-0636643.

\section{Preliminary Work: Signature Requirements}

We use the following theorem to determine the signature of our pretzels. It is Theorem 6 in \cite{gordon2}.

\begin{thm}[Gordon-Litherland]
For any checkerboard-coloured diagram $D$ of the knot $K$ with associated Goeritz matrix $G(D)$,
$$\sigma(K)=\sgn(G(D))-\mu(D),$$
where $\mu(D)$ is the correction term of the diagram.
\end{thm}

As a brief note before continuing, because we will always be using the same diagram for our pretzels, we will write $G(K)$ and $\mu(K)$ with this diagram understood. Moreover, when we speak of the determinant of $K$, this will always be positive. The determinant of $G(K)$, however, can be signed, and this is important for our later classification. Thus, in general, $\det K = \abs{\det G(K)}$.

\begin{figure}[h]
\centering
\subfigure[]{
\includegraphics[clip = true, scale = .35]{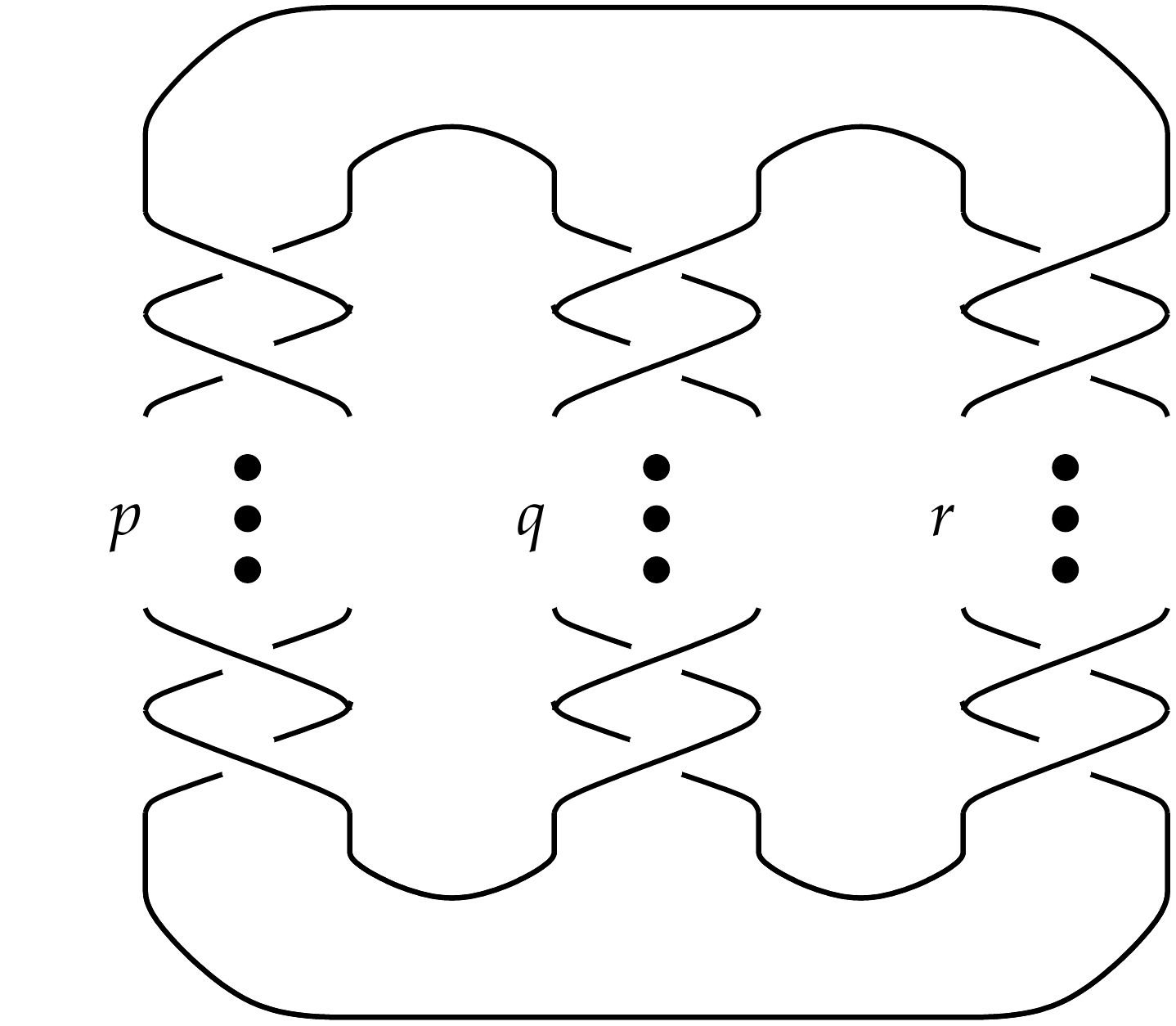}
\label{fig:Pretzel1}
}\hskip1in
\subfigure[]{
\includegraphics[clip = true, scale = .35]{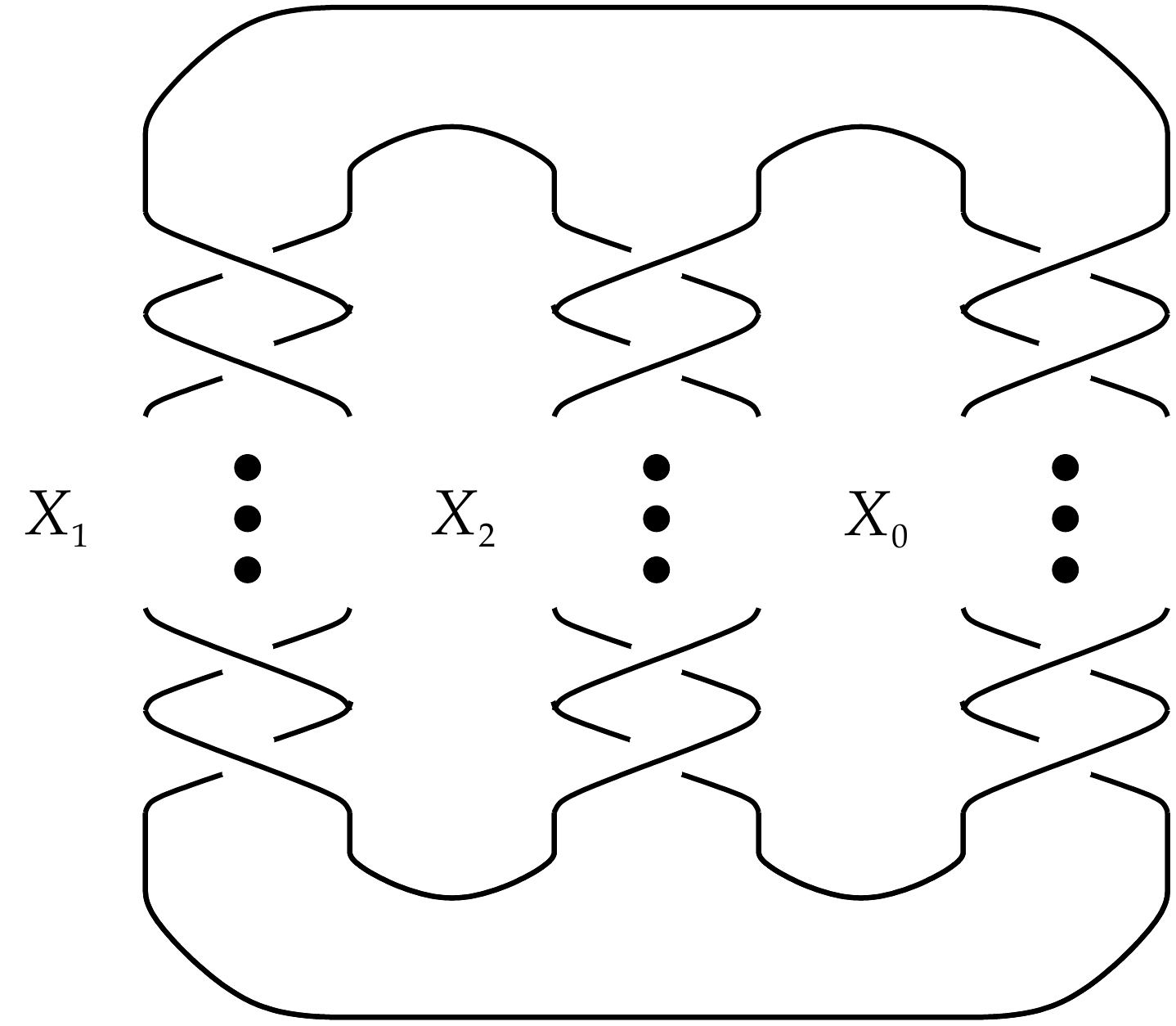}
\label{fig:Pretzel2}
}
\caption{(a) Pretzel knot $P(p,q,r)$, where $p<0$, and $q,r>0$; and (b) the knot $P(p,q,r)$ with a checkerboard colouring.}
\end{figure}

With these conventions in mind, we apply the theorem above to a standard diagram of the knot $P(p,q,r)$, Figure \ref{fig:Pretzel1}, where $p$ and $q$ are odd, and $r$ is even.  By shading and labelling the three regions of Figure \ref{fig:Pretzel1} as marked by the $X_i$ in \ref{fig:Pretzel2}, we obtain the Goeritz matrix of $K$:\\

\begin{center}$G(K)=
\left(\begin{array}{cc} p+r&-p\\
-p&p+q\end{array}\right).$
\end{center}
\vskip.15in

Note that the matrix $G(K)$ is $2\times2$, and therefore $\sgn (G(K))\in\{-2,0,2\}$.  In particular if $u(K)=1$, then $\mu$ is restricted to $\{-4,-2,0,2,4\}$. According to \cite{gordon2}, the correction term $\mu$ is the sum of the crossing numbers in the $p$ and $q$ columns. Since $|p|,|q|,|r|>1$, if $p$ and $q$ are both the same sign then $\abs{\mu}=\abs{p}+\abs{q}\geq6$, a contradiction. So without loss of generality, take $p>0$ and $q<0$. Furthermore the reflection invariance of unknotting number allows us to assume $r=2m>0$ (we ignore $m = 0$ for reasons below). Relabel the knot $K=P(p,q,r)$ as $K=P(k,-k+n,2m)$, where $m>0$, $k>1$ odd, and $n\in\{-4,-2,0,2,4\}$. The Goeritz matrix thus becomes:

\begin{center}$G(K)=
\left(\begin{array}{cc} k+2m&-k\\
-k&n\end{array}\right),$\end{center}

\noindent which implies

$$\det(G(K))=-k^2+kn+2mn.$$

If we consider the case when $n = -4$, then we compute easily that $\sigma(K) = 4$, which is not within the range for unknotting number one. Also, if $n = 4$ and $\det G(K) <0$, then $\sigma(K) = -4$, and we can rule this possibility out for the same reason. Hence, we have five remaining cases:

\begin{center}
\begin{tabular}{c||ccccc}
Case &$n$ &$\det G(K)$ &$\sigma(K)$\\
\hline
\hline
1 &$-2$ & &$2$\\
2 &$0$ & &$0$\\
3a &$2$ &$<0$ &$-2$\\
3b &$2$ &$>0$ &$0$\\
4 &$4$ &$>0$ &$-2$\\
\end{tabular}
\end{center}

The first two of these are treated in Sections \ref{s:-2}, \ref{s:0a}, and \ref{s:0b}, while the remaining cases are the domain of our concluding remarks in Section \ref{s:2,4}.

As a final remark in this section, although we mentioned that we will only be considering $m > 0$, for completeness we can dismiss $m = 0$ immediately. In this instance, $P(p,q,0) = T(p,2)\#T(q,2)$, and since unknotting number one knots are prime (see Scharlemann \cite{ScharlemannPrime} or Zhang \cite{Zhang}), it follows that one of $p,q=\pm 1$. Then, as mentioned in the Introduction (via \cite{Mrowka} and \cite{Rasmussen}), since the signature of torus knots is a tight bound on $u(K)$, and $\sigma(T(k,2)) = \half (k-1)$ for $k \geq 1$, we obtain the following result.

\begin{lem}
If $K = P(p,q,0)$ and $u(K) = 1$, for $p,q$ odd, then $pq=\pm3$.
\end{lem}

\section{The Case $p+q=-2$}
\label{s:-2}

In this section we consider $K = P(k,-k-2,2m)$ where $k$ odd, $k\geq1$, and $m>0$.  Our method has two main ingredients: the signed Montesinos theorem and Greene's application of Donaldson's diagonalisation theorem to $u(K) = 1$. Our main theorem is the following.

\begin{thm}
\label{-2}
Suppose that $k,m > 0$ and $k$ is odd. Then $P(k,-k-2,2m)$ has unknotting number one if and only if $k =1$.
\end{thm}

Recall that when $K = P(k,-k-2,2m)$ we have $\sigma(K) = 2$. Since the conclusion to the above theorem is that $k = 1$, we aim to prove it by establishing that if $k \geq 3$ then $u(K) \neq 1$. In the case $k = 1$, we can change any crossing in the central column to obtain $P(1,-1,2m)$, which is manifestly the unknot. The only knots with $p+q=-2$ not treated, then, are those of the form $P(-1,-1,2m)$, and the fact that $u(K) = 1$ is clear in that instance.

As mentioned, our first ingredient is the ``signed" version of the Montesinos theorem (see Proposition 4.1 of \cite{GreeneBraid}).
\begin{thm}[Signed Montesinos]\label{thm:montesinos} Suppose that $K$ is a knot that is undone by changing a negative crossing (so $\sigma(K)=0,2$). Then $\Sigma(K)=S^3_{-\epsilon D/2}(\kappa)$ for some knot $\kappa\subset S^3$, where $D=\det(K)$, and $\epsilon=(-1)^{\half\sigma(K)}$.  In particular, $-\Sigma(K)=S^3_{\epsilon D/2}(\overline{\kappa})$ bounds a smooth, simply connected, 4-manifold $W_K$ with $\epsilon$-definite intersection form $-\epsilon R_n$, where
$$R_n= \left(
  \begin{array}{cc}
               -n&1\\ 
               1&-2\\ 
  \end{array}
\right)$$ 
and $D=2n-1$.
\end{thm}

As we have $\sigma(K) = 2$, if $u(K) = 1$ then $-\Sigma(K)$ bounds a negative-definite 4-manifold $W_K$ from Theorem \ref{thm:montesinos}.  In order to use Donaldson's Theorem A we need another $4$-manifold which is bounded by $\Sigma(K)$, call this $X_K$, with intersection form $Q_K$, so that we can glue them together to obtain a closed manifold $X = X_K \cup_{\Sigma(K)} W_K$. Since the boundary $\Sigma(K)$ is a rational homology 3-sphere, $Q_K \oplus R_n$ embeds into the intersection form $Q_X$ of $X$, as can be seen from the Mayer-Vietoris sequence (see \cite{GreeneBraid}). As $W_K$ is simply connected (by Theorem \ref{thm:montesinos}), if $X_K$ is simply connected then so too is $X$. We are now ready to use Donaldson's Theorem A (see \cite{Donaldson}).

\begin{thm}[Donaldson]\label{thm:donaldson} Let $X$ be a closed, oriented, simply connected, smooth 4-manifold. If the intersection form $Q_X$ is negative-definite, then $Q_X$ diagonalises over the integers to $-\id$.
\end{thm}

In the light of the above comments, we have the following corollary.

\begin{cor}\label{diag}
If $X = X_K \cup_{\Sigma(K)} W_K$ is simply connected and negative-definite, then there exists an integral matrix $A$ such that
\begin{equation}
\label{A}
-AA^t = Q_K \oplus R_n.
\end{equation}
\end{cor}

Thus, if we can show that there does not exist an $A$ satisfying \eqref{A}, then $u(K) > 1$ (or $K$ is the unknot). The first question, then, is how to find $X_K$, and for this we use plumbing. A good reference for the following section is \cite{Gompf}.

\subsection{Plumbings}
\label{plumbings}

Let $G$ be a vertex-weighted simple graph with vertex set $V(G)$ and labels $w(v)$ on each $v \in V(G)$. In general, we take $w(v)<0$ since we are mainly concerned with negative-definite manifolds. To construct a 4-manifold $X=X(G)$ from $G$, take the 2-disc bundle $B(v)$ over $S^2$ of Euler number $w(v)$ for each $v \in V(G)$, and plumb $B(v)$ and $B(v^\prime)$ if and only if $v$ and $v^\prime$ are adjacent in $G$. This manifold $X$ has free $H_2(X)$, generated by the homology classes of spheres $S_v$ corresponding to the vertices. We will write these as $[S_v]$.

Supposing that $G$ is a tree, then $X(G)$ is simply connected. The manifold $Y = Y(G) = \partial X$ is given by a Kirby diagram of unknots, linked geometrically according to the weighted adjacency matrix for $G$ (so that the slopes on the components are the weights of the corresponding vertices). The intersection form $Q$ for $X(G)$ is then also the adjacency matrix for $G$. Explicitly, we have $\left<[S_v], [S_v]\right> = w(v)$ for each vertex, and $\left<[S_v], [S_{v^\prime}]\right> = 1$ if the two distinct vertices are connected by an edge, zero otherwise.

Since $\Sigma(K)$ is a Seifert fibred space, it has the surgery presentation given in Figure \ref{fig:firstsurgery}. Hence, we can obtain a plumbing with boundary $\Sigma(K)$ using the corresponding graph. However, for what will follow, this 4-manifold is insufficient since it is not negative-definite. Instead, we use the alternative presentation in Figure \ref{fig:secondsurgery} and the corresponding plumbing shown in Figure \ref{fig:goodplumbing}.

\begin{center}
\begin{figure}[h]
\centering
\subfigure[]{
\includegraphics[scale = .45]{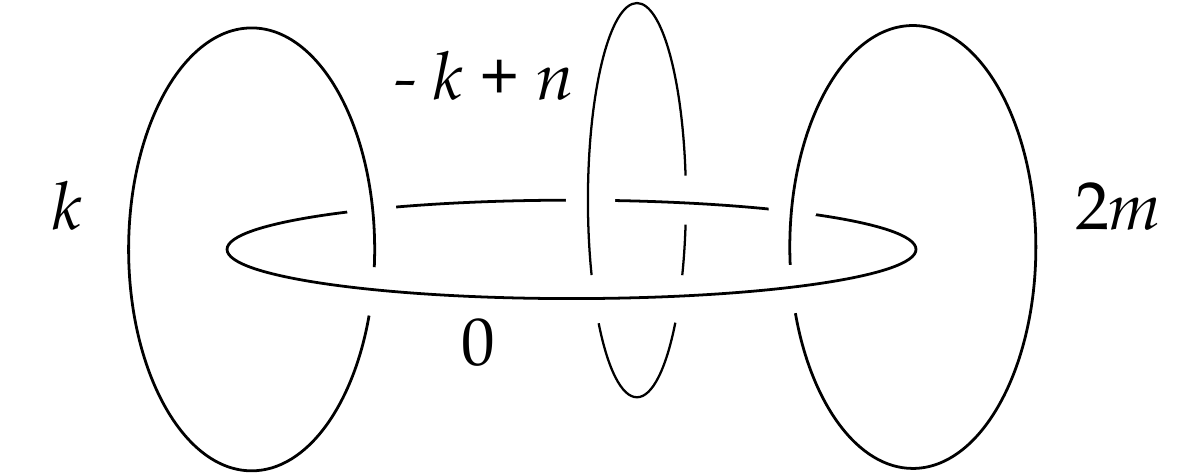}
\label{fig:firstsurgery}
}\hskip1in
\subfigure[]{
\includegraphics[scale = .45]{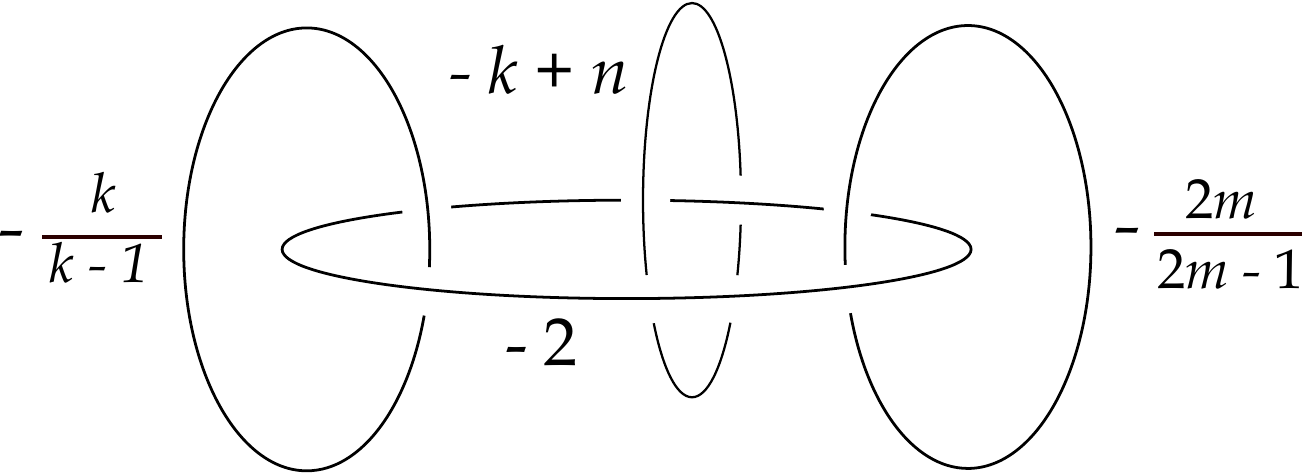}
\label{fig:secondsurgery}
}
\caption{(a) Kirby diagram for $\Sigma(K)$; and (b) alternative Kirby diagram for $\Sigma(K)$.}
\end{figure}
\end{center}

\begin{figure}[h]
\centering
\includegraphics[trim = 0in 4in 0in 5in, scale = 0.5]{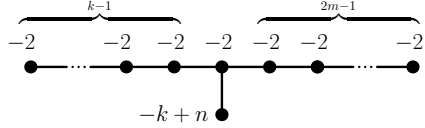}
\caption{A graph $G$ for a plumbing $X_K$ with boundary $\Sigma(K)$. The vertices are labelled from left to right along the top row as $v_1$ to $v_{k+2m-1}$; the final vertex is $v_{k+2m}$.}
\label{fig:goodplumbing}
\end{figure}

Two things must be checked about the plumbing in Figure 3. First, that the boundary is $\Sigma(K)$. This is easily done once we observe that
$$\frac{k}{k-1}=\overset{k-1}{\overbrace{[2,2,\dots,2]}} \qquad \frac{2m}{2m-1}=\overset{2m-1}{\overbrace{[2,2,\dots,2]}}.$$
Here $[a_1,\dots,a_\ell]$ denotes the Hirzebruch-Jung continued fraction. Therefore, as $\partial X_K$ has a Kirby diagram given by unknots linked according to $G$, we can slam-dunk these unknots along the two long arms to obtain the diagram in Figure \ref{fig:secondsurgery}. Performing $+1$ twists around each of the two non-integrally framed unknots will then recover Figure \ref{fig:firstsurgery}.

The second requirement is that $Q_X$, the intersection form of $X_K$, be negative-definite. The key component here is Sylvester's criterion.

\begin{lem}[Sylvester]
Let $M$ be a square matrix and $M_i$ its upper $(i\times i)$-submatrix. Then $M$ is negative-definite if and only if the sign of $\det M_i$ is $(-1)^i$ for all $i$.
\end{lem}

Observing that the upper submatrices of $Q_X$, with the exception of the total matrix, are all $-2$ along the diagonal and $1$ in the spots adjacent to the diagonal, the determinants are $(-1)^i(i+1)$. It is then also easy to see that $\det Q_{X_K} = \det G(K) < 0$, and as the rank of $Q_{X_K}$ is odd, we are done.

At this point in the proceedings, we form $X=X_K\cup_{\Sigma(K)}W_K$, which is closed. Unfortunately, however, for this choice of $X$ there always exists an $A$ satisfying \eqref{A}. To get around this problem, we mimic the work of Greene \cite{GreeneBraid}, in which Heegaard Floer homology is used to impose a certain structure on $A$. In order to explain this, we review some Heegaard Floer homology.

\subsection{Correction Terms and Sharpness}
\label{correctsharp}

Ozsv\'ath and\ Szab\'o have shown in \cite{OSAbsolute} that the Heegaard Floer homology of a rational homology sphere $Y$ is absolutely graded over $\mathbb{Q}$. They also give a definition of correction terms, $d(Y,\mathfrak{t})$, which are the minimally graded non-zero part in the image of $HF^\infty(Y,\mathfrak{t})$ inside $HF^+(Y,\mathfrak{t})$. These are strongly connected to the topology of 4-manifolds with $Y$ as boundary, for any such negative-definite, smooth, oriented $X$ which has an $\mathfrak{s} \in \Spin(X)$ such that $\mathfrak{s}\vert_Y = \mathfrak{t}$ must satisfy
\begin{equation}
\label{eq:corterm}
\mathfrak{c}_1(\mathfrak{s})^2 + b_2(X) \leq 4d(Y,\mathfrak{t}).
\end{equation}
A rational homology 3-sphere $Y$ is an \emph{$L$-space} if $\rank\widehat{HF}(Y)=\abs{H_1(Y)}$. Furthermore, a \emph{sharp} 4-manifold $X$ with $L$-space boundary $Y$ is defined by the property that for every $\mathfrak{t}\in\Spinc(Y)$ there is some $\s\in\Spinc(X)$ with $\s\vert_Y=\mathfrak{t}$ that attains equality in the bound \eqref{eq:corterm}.

We are now able to present Greene's theorem. It is proved in \cite{GreeneBraid}. (Our convention for $L(p,q)$ should be taken as the $-p/q$ surgery on the unknot.)

\begin{thm}[Greene] \label{greene} Suppose $K$ is a knot in $S^3$ with unknotting number one such that either (i) $\sigma(K) = 0$ and $K$ can be undone by changing a positive crossing, or (ii) $\sigma(K) = 2$. Suppose also that $\Sigma(K)$ is an $L$-space and
$$d(\Sigma(K),0) = -d(L(\det K,2),0).$$
Then if $X_K$ is a smooth, sharp, simply connected 4-manifold with rank $r$ negative-definite intersection form $Q_K$, and $X_K$ is bounded by $\Sigma(K)$, there exists an integral matrix $A$ such that $-AA^T=Q_K\oplus R_n$, and $A$ can be chosen such that the last two rows are $(x_{r+2}, \dots, x_3, 1, 0)$ and $(0,\dots,0, -1,1)$.  Furthermore the values $x_3,\dots,x_{r+2}$ are non-negative integers and obey the condition
\begin{equation}
x_3\leq 1,\,\,x_i\leq x_3+\dots+x_{i-1}+1\,\text{ for } 3<i<r+2,
\end{equation}
and the upper right $r \times r$ matrix of $A$ has determinant $\pm 1$.
\end{thm}

We have already shown that $X_K$ is simply connected and negative-definite, so what remains is to check that $\Sigma(K)$ is an $L$-space, and that $X_K$ is sharp. For the $L$-space condition, we refer to Section 3.1 of \cite{PretLSpace}, which immediately yields our result. To show that $X_K$ is sharp, we use Theorem 1.5 in \cite{OSPlumbed}. Since the negative-definite plumbing diagram has one ``overweight'' vertex (or ``bad'' in the sense of \cite{OSPlumbed}), it follows that $X_K$ is sharp.

The remaining condition on the correction terms is more difficult to check, and will require some more sophisticated knowledge of the Heegaard Floer homology of plumbed manifolds. Since this is material best presented in Section \ref{s:0b}, we ask the reader to suspend his or her disbelief until Lemma \ref{check}.

\subsection{The Proof of Theorem \ref{-2}}

We are now ready to prove Theorem \ref{-2}. To do this, we show that the $A$ described in Theorem \ref{greene} does not exist when $k \geq 3$. We begin by writing down $Q_K\oplus R_n$:

$$Q_K\oplus R_n= \left(
  \begin{array}{cccccccc|cc}
    -2         & 1   &   0        && \dots    &&&  0 &&\\
    1          & -2  &   1        &&              &&& &&  \\
     0         &  1  &  -2        &&  &&& &&  \\
                &       &               &&     \ddots           &&&   1&&\\
        \vdots &      &  &\ddots& \ddots  &&&&& \\
                    &       &            &&               &&1& &&   \\         
                   &       &            &&               &1&-2& 0 &&  \\  
         0       &      &               &1&           &&0&-k-2&&\\
         \hline
               &      &               &&           &&&&-n&1\\ 
               &      &               &&           &&&&1&-2\\ 
  \end{array}
\right).$$

Here the $(k+2m,k)$ and $(k,k+2m)$ entries are both $1$. It will be helpful to label the rows of $A$ as $v_1,\dots,v_{k+2m+2}$. Observe that $M_{i,j}=-(AA^T)_{i,j}=-v_i\cdot v_j$. Since $|v_i\cdot v_i|=2$ for $i\neq k+2m$ and $k+2m+1$, each row of $A$ (except rows $k+2m$ and $k+2m+1$) has two non-zero entries, each of magnitude 1. Without loss of generality set $v_1=(1,-1,0,\dots,0).$  Making this choice and applying the two row conditions from Theorem \ref{greene}, the remaining rows must take the following form (after permuting the columns of $A$):

$$A= \left(
  \begin{array}{cccccc|cc}
 1& -1 &  &  &  &  &  &   \\
  &  1 & -1&  &  &  &  &   \\ 
  &   &  \ddots & \ddots &  &  &  &   \\
  &   & &  &  &  &  &   \\ 
  &   &  &  &  &  &  &   \\ 
  &   & &  & 1 & -1 &  &   \\ 
 * &  * &\dots &  \dots&*  &*  & * & *  \\ 
   \hline
   *& *  &\dots & \dots & * & * & 1 &   \\ 
  &   &  &  &  &  &-1  &1   \\   
   
  \end{array}
\right).$$
This implicitly requires us to note that $k + 2m - 1 \geq 4$, and so the first $k + 2m - 1$ rows cannot have more than one non-zero entry in the same spot.

Next let $A_{k+2m+1,1}=\alpha$. Since $v_i\cdot v_{k+2m+1}=1$ for $i=1,2,\dots,k+2m-1$, each of the first $k+2m$ entries along the $(k+2m+1)$-th row all equal $\alpha$:

$$A= \left(
  \begin{array}{cccccc|cc}
 1& -1 &  &  &  &  &  &   \\
  &  1 & -1&  &  &  &  &   \\ 
  &   &  \ddots & \ddots &  &  &  &   \\
  &   & &  &  &  &  &   \\ 
  &   &  &  &  &  &  &   \\ 
  &   & &  & 1 & -1 &  &   \\ 
 * &  * &\dots &  \dots&*  &*  & * & *  \\ 
   \hline
   \alpha& \alpha  &\dots & \dots & \alpha & \alpha & 1 & 0  \\ 
  &   &  &  &  &  &-1  &1   \\   
  \end{array}
\right).$$

According to Theorem \ref{greene}, $\alpha=0$ or $1$.
\begin{enumerate}
\item If $\alpha=0$, $v^2_{k+2m+1}=n=\half(\det(K)+1)=1$, and therefore $\det(K)=1$.  One can use the Goeritz matrix to show that $\det P(k,-k-2,2m)=k^2+2k+4m$, so clearly $\alpha\neq0$.
\item If $\alpha=1$, then $v^2_{k+2m+1}=n=\half(\det K+1)=k+2m+1$.  This only happens if $k^2=1$, which contradicts our assumption that $k\geq 3$.
\end{enumerate}
The reader will note, as before, that knots of the form $P(1,-3,2m)$ have unknotting number one for all integral $m$. We have thus completed the first piece of our classification.

\section{The Case $p+q =0$: First Results}
\label{s:0a}

In this section, we tackle the knots $K = P(k,-k,2m)$. Our general method is as follows: we first pin down the value of $m$ using the Alexander module, finding that $m = 1$, then employ Heegaard Floer homology to deduce the value of $k$. One naturally wonders if the methods of the previous section will help us in this endeavour, but unfortunately the previous method only allows us to identify the sign of the crossing change involved.

Our ultimate goal over this section and the next is the following theorem.

\begin{thm}
\label{0}
Suppose that $k, m > 0$ and $k$ is odd. Then $P(k,-k,2m)$ has unknotting number one if and only if $k = 3$ and $m = 1$.
\end{thm}

The determinant of $P(k,-k,2m)$ is always $k^2$. Hence, in Theorem \ref{thm:montesinos}, we can always take $D = k^2$, and so $n = \tfrac{k^2+1}{2}$.

\subsection{The Alexander Module}

Recall that we can construct the infinite cyclic cover $X_\infty$ of a knot. This has a deck transformation group $\mathbb{Z}$, generated by some element $t$. Then $H_1(X_\infty;\mathbb{Z})$ is a $\mathbb{Z}[t,t^{-1}]$-module $A$, called the \emph{Alexander module}, from which much topological information can be extracted. This is done via the $r^\text{th}$ \emph{elementary ideal}, denoted $A_r$, which is the ideal of $\mathbb{Z}[t,t^{-1}]$ spanned by the $(n - r + 1)\times(n-r+1)$-minors of any $n\times n$ presentation matrix for $A$.

From Nakanishi \cite{NakanI}, in the form cited in Lickorish \cite{Lickorish}, we know that the Alexander module can bound the unknotting number. For our purposes, we present the following definition-theorem (see Theorem 7.10 of Lickorish \cite{Lickorish}).

\begin{thm}[Unknotting via Alexander module]
\label{Alexander}
Suppose that $V$, an $n \times n$ matrix, is a Seifert matrix for $K$ in $S^3$. Then the Alexander module of $K$ is presented by the matrix $A = tV - V^t$. Moreover, if $\mathbb{Z}[t,t^{-1}]/A_r \neq 0$, it follows that $u(K) \geq r$.
\end{thm}

Using this, we can now prove the following lemma.

\begin{lem}
Suppose $k \geq 3$, $m > 0$, and $k$ is odd. Then if $P(k,-k,2m)$ has unknotting number one, $m = 1$.
\end{lem}
\begin{proof}
We take the following Seifert surface for our pretzels, $P(k,-k,2m)$. The curves are indexed starting with the leftmost column of loops, smallest to largest, followed by the same labelling in the next column. For the last two curves, we take the big loop around the hole, then the loop crossing the ``bridge''. As regards orientations, the different shadings represent differences in orientation.
\begin{center}
\includegraphics[width=0.6\textwidth]{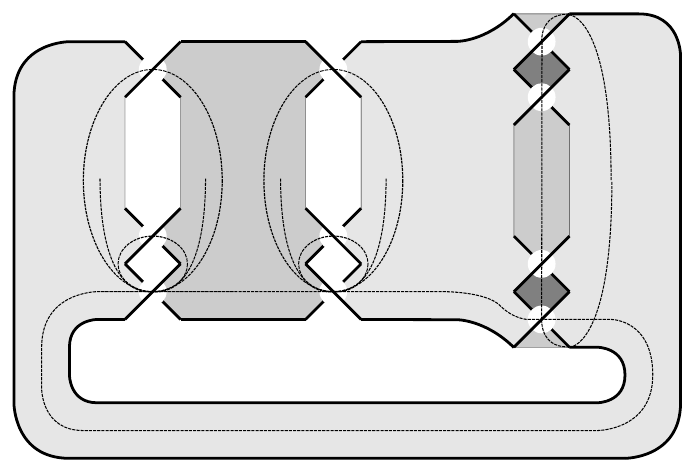}
\end{center}
From this, we construct a Seifert matrix for $P(k,-k,2m)$ of the form
$$V = \begin{pmatrix}X_k&0&0&0\\0&-X_k&0&0\\\mathbf{1}&-\mathbf{1}&0&0\\0&0&1&m\end{pmatrix},$$
where $X_k$ is the $(k-1) \times (k-1)$ lower triangular matrix of 1's, and $\mathbf{1}$ is a suitably sized row of 1's.

Consequently, the Alexander module is presented by
$$A = \begin{pmatrix}M_k&0&-\mathbf{1}^t&0\\0&-M_k&\mathbf{1}^t&0\\\mathbf{t}&-\mathbf{t}&0&-1\\0&0&t&m(t-1)\end{pmatrix},$$
from which we can compute the relevant minors. Here, $M_k = tX_k - X_k^t$, and $\mathbf{t}$ is a row with all entries $t$.

We claim, that for $k \geq 3$, the second elementary ideal, $A_2$, generated by these minors (in $\mathbb{Z}[t,t^{-1}]$), is precisely given by
$$A_2 = \left<\sum_{i=0}^{k-1}(-1)^it^{k-1-i}, m(t-1) \right>.$$
For the moment, we assume this, and call the first polynomial $\mathcal{P}_k(t)$. Then we can show that $A_2 = \mathbb{Z}[t,t^{-1}]$ if and only if $m = 1$, since $k$ is odd. Indeed, the quotient $\mathbb{Z}[t,t^{-1}]/\left<\mathcal{P}_k(t)\right>$ is the $\mathbb{Z}$-module consisting of all integral Laurent polynomials with the form
$$a_{k-2}t^{k-2} + a_{k-3}t^{k-3} + \dots + a_1t+a_0,$$
together with their unit multiples (that is, multiples of $t^n$ for $n$ an integer). These are forced to be zero in $A/A_2$ if and only if they fall in the ideal $\left<m(t-1)\right>$. In particular, we require all $a_i$ to be divisible by $m$. This statement then implies $m = 1$.

When $m = 1$, observe that $\mathcal{P}_k(t)$ is in fact, for $k$ odd,
$$\mathcal{P}_k(t) = t^{k-1} - t^{k-3}(t-1) - t^{k-5}(t-1) - \dots - (t-1),$$
which means that in the quotient $\mathbb{Z}[t,t^{-1}]/\left<m(t-1)\right>$, the polynomial is a unit, since $\mathcal{P}_k(t) \equiv t^{k-1}$, whence $\mathbb{Z}[t,t^{-1}]/A_2 = 0$. Hence there is no obstruction to unknotting number one, since the theorem guarantees only that $u(K) \geq 1$.

What remains then is to check our claim. As a first step, we can compute the determinant of $M_k$, which goes as follows. Here, for a row vector $\mathbf{v}$, we use the notation $\mathbf{v}^*$ to indicate a square matrix with each row $\mathbf{v}$.
\begin{align*}
\det M_k &= \det\begin{pmatrix}
t-1 &-1 &-1&\dots &-1 \\
t &t-1 &-1& \dots &-1 \\
t &t & t-1 &\dots &-1 \\
\vdots &\vdots &\vdots&\ddots &\vdots \\
t & t &t &\dots &t-1
\end{pmatrix}\\
&= t^{-1}\det \begin{pmatrix}
t^2 &0 &0&\dots &-1 \\
t &t-1 &-1& \dots &-1 \\
t &t & t-1 &\dots &-1 \\
\vdots &\vdots &\vdots&\ddots &\vdots \\
t & t &t &\dots &t-1
\end{pmatrix}\\
&= t\det M_{k-1} +(-1)^{k-1}t^{-1}\det \begin{pmatrix}\mathbf{t}^t &M_{k-2}\\ t & \mathbf{t}\end{pmatrix}\\
&=t\det M_{k-1} + (-1)^{k-1}\det\begin{pmatrix}0 & M_{k-2}-\mathbf{t}^* \\ 1 &\mathbf{1} \end{pmatrix}\\
&=t\det M_{k-1} + \det(M_{k-2}-\mathbf{t}^*)\\
&= t\det M_{k-1} + (-1)^{k-1}.
\end{align*}

From this recurrence we can see that $\det M_k = \mathcal{P}_k(t)$. Now it is not hard to see that for any other minor, supposing that the $m(t-1)$ entry remains, we can expand down the final column or row. This yields two terms, one that contains $m(t-1)$ as a factor, and the other of which is the determinant of a block diagonal matrix, one factor of which is either $\det(M_k)$ or $\det(-M_k)$, both of which are $\mathcal{P}_k(t)$ up to sign. The remaining case, when $m(t-1)$ is removed, is a calculation very much like that following this paragraph, and therefore has $\mathcal{P}_k(t)$ as a factor. What remains to be done, then, in order to prove that $A_2$ is spanned by these two key polynomials is to ensure that they are both actually in the ideal.\ This is proved by the following two example minors.

First, we delete the first row and final column:
\begin{align*}\det A^{1,2k} &= \det\begin{pmatrix}
\mathbf{t}^t&M_{k-1}&0&-\mathbf{1}^t\\0&0&-M_k&\mathbf{1}^t\\t&\mathbf{t}&-\mathbf{t}&0\\0&0&0&t
\end{pmatrix}\\
&=t \det \begin{pmatrix}\mathbf{t}^t &M_{k-1} &0\\0 & 0 &-M_k \\t &\mathbf{t} &-\mathbf{t}\end{pmatrix}\\
&=-t^2\det \begin{pmatrix}M_{k-1}-\mathbf{t}^* &\mathbf{t}^*\\0 &-M_k \end{pmatrix}\\
&=(-1)^kt^2\det(M_{k-1}-\mathbf{t}^*)\det M_k\\
&=t^2\mathcal{P}_k(t).
\end{align*}

The last equality uses our previous calculation, and the fact that $M_{k-1} - \mathbf{t}^*$ is an upper-triangular matrix with all its $(k-2)$ diagonal entries being $-1$. Since $t^2$ is a unit, we know that $\mathcal{P}_k(t)$ is in $A_2$. We now check that $m(t-1)$ is too, as evidenced by the following minor.
\begin{align*}
\det A^{1,k}
&= \det \begin{pmatrix}
\mathbf{t}^t&M_{k-1}&0&-\mathbf{1}^t&0\\
0&0&\mathbf{1}&1 &0 \\
0&0&-M_{k-1}&\mathbf{1}^t&0 \\
t&\mathbf{t}&-\mathbf{t}&0&-1 \\
0&0&0&t&m(t-1)
\end{pmatrix}\\
&= m(t-1)\det\begin{pmatrix}
\mathbf{t}^t&M_{k-1}&0&-\mathbf{1}^t\\
0&0&\mathbf{1}&1\\
0&0&-M_{k-1}&\mathbf{1}^t\\
t&\mathbf{t}&-\mathbf{t}&0
\end{pmatrix}\\
&=m(t-1)\det\begin{pmatrix}
0&M_{k-1}-\mathbf{t}^*&\mathbf{t}^*&-\mathbf{1}^t\\
0&0&\mathbf{1}&1\\
0&0&-M_{k-1}&\mathbf{1}^t\\
t&\mathbf{t}&-\mathbf{t}&0
\end{pmatrix}.
\end{align*}

The last matrix determinant is then manipulated as
$$-t\det(M_{k-1}-\mathbf{t}^*)\det\begin{pmatrix}\mathbf{1} &1\\-M_{k-1} & \mathbf{1}^t\end{pmatrix}=(-1)^{k-1}\det\begin{pmatrix}M_{k-1} & -\mathbf{1}^t \\ \mathbf{t} &t\end{pmatrix},$$
and this in turn is almost $M_k$. The RHS is in fact, up to sign,
$$\det M_k - (t-1)\det M_{k-1} + t \det M_{k-1} = \det M_k + \det M_{k-1} = t^{k-1}.$$
It follows that $m(t-1)$ is in $A_2$, at last completing our proof.
\end{proof}

\subsection{Donaldson Diagonalisation and $\Sigma(k,-k,2)$}
\label{diagdonzero}

As foreshadowed, we can try to mimic the work in Section \ref{s:-2}. However, since in this case the signature of $K$ vanishes, the only progress we can make here is to pin down the sign of the unknotting crossing. This is due to problems gluing the pieces of our closed manifold $X$ since the orientations must be compatible. This information, however, will be relevant in the next section on the Heegaard Floer homology obstruction.

\begin{lem}
\label{sign}
Suppose $k \geq 3$ is odd. Then if $K = P(k,-k,2)$ has unknotting number one, it is undone by changing a negative crossing.
\end{lem}
\begin{proof}
Suppose that $K$ is undone with a positive crossing. Then, $\overline{K}$ is undone by changing a negative crossing. Hence, in the signed Montesinos theorem (Theorem \ref{thm:montesinos}), $\Sigma(\overline{K}) = S^3_{-D/2}(\kappa)$, where $\kappa$ is a knot in $S^3$ and $D = \det \overline{K} = \det K$, and so $-\Sigma(K)$ bounds a negative-definite 4-manifold with intersection form $R_n$.

The plumbing in Figure \ref{fig:goodplumbing} is negative-definite, proved exactly analogously to the case treated in Section \ref{plumbings}. As before, $\Sigma(K)$ is an $L$-space via Section 3.1 of \cite{PretLSpace}. The fact that $d(\Sigma,0) = d(L,0) = 0$ is, also as before, delayed until Lemma \ref{check}.

Knowing that Theorem \ref{greene} is applicable, we glue these two 4-manifolds together, as before, and the matrix $A$ should appear as
$$\left( \begin{array}{ccccccc|cc}
1 &-1 & & & & & & &\\
 &1 &-1 & & & & & &\\
 & &1 &-1 & & & & &\\
 & & &\ddots &\ddots & & & &\\
 & & & &1 &-1 & & &\\
 & & & & &1 &-1 & &\\
a &a &a &\dots &a &b &b &c &c\\
\hline
d &d &d &\dots &d &d &d &1 &\\
 & & & & & & &-1 &1
\end{array} \right).$$
Denote the rows by $v_i$, with a total of $k+4$ rows. Then $v_k\cdot v_{k+2} = -1$, so $b = a-1$. Then $v_{k+2}\cdot v_{k+2} = k$ implies
\begin{equation}
\label{epsone}
ka^2 + 2(a-1)^2 + 2c^2 = k
\end{equation}
whence we must have $a = 0,1$ (else the LHS is too big). We split the cases:
\begin{enumerate}
\item If $a = 0$, then from \eqref{epsone} we have $2c^2 + 2 = k$. This is nonsense for parity reasons.
\item If $a = 1$, then $c = 0$ (from \eqref{epsone}). Then $v_{k+2}\cdot v_{k+3} = 0$ tells us $kd = 0$, whence $d = 0$. The fact that $v_{k+3}\cdot v_{k+3} = n$ yields up $n = 1$, so $k^2 = 1$, contradicting $k \geq 3$.
\end{enumerate}
This completes the proof.
\end{proof}

\section{Heegaard Floer Homology of $\Sigma(P(k,-k,2))$}
\label{s:0b}

To complete the work started in the previous section we now compute the graded Heegaard Floer homology of $\Sigma(P(k,-k,2))$. The key technology for this is found in Ozsv\'ath and Szab\'o \cite{OSPlumbed}, where the two authors present a combinatorial algorithm for determining the Heegaard Floer homology of plumbed three-manifolds (such as small Seifert fibred spaces, as we have here).

Before we can explain why the Heegaard Floer homology is relevant, however, it is good to streamline some notation. Define $D:= \det P(k,-k,2)=k^2$ and write $\Sigma := \Sigma(P(k,-k,2))$ and $L:=L(D,2)$. The integer $n$ should be defined by $D = 2n-1$ and since $D \equiv 1 \mod 4$, we set $n = 2s+1$. We remind the reader that we implicitly only care about $k \geq 3$, since $k = 1$ yields the unknot.

Now the obstruction to $u(K)=1$, taken from Theorem 4.1 of \cite{OSUnknot}. We present only the half of this theorem where $D\equiv 1 \mod 4$, since that is all we need.

\begin{thm}
\label{OSunknot}
If $\kappa$ is a knot in $S^3$ such that $S^3_{-D/2}(\kappa)$ is an $L$-space, where $D \equiv 1 \mod 4$, and if
\begin{equation}
\label{zero}
d(S^3_{-D/2}(\kappa),0) = d(L,0),
\end{equation}
then for $i = 0, 1,\dots, s$,
\begin{equation}
\label{obs}
d(S^3_{-D/2}(\kappa),i) - d(L,i) = d(S^3_{-D/2}(\kappa),2s-i) - d(L,2s-i),
\end{equation}
where the labelling on the $\Spinc$-structures is by $\half \mathfrak{c}_1$.
\end{thm}

The symmetries exhibited in \eqref{obs} give our obstruction to unknotting number one as follows. By Lemma \ref{sign} we know that $K = P(k,-k,2)$ must be undone by changing a negative crossing. Therefore, applying Theorem \ref{thm:montesinos}, $\Sigma(K) = S^3_{-D/2}(\kappa)$ for some knot $\kappa$. We already know that $\Sigma(K)$ is an $L$-space. Hence, provided that we can establish that \eqref{zero} holds, the equations \eqref{obs} will give our obstruction: we shall show that if $k \geq 5$, at least one of them must fail.

If the reader is wondering why we do not use the full power of Theorem 1.1 of \cite{OSUnknot}, the reason is that the conditions on positive and even matchings are not strong enough to obstruct our pretzels. The symmetry condition, however, is, and this is essentially just Theorem \ref{OSunknot}.

\subsection{Correction Terms of $\Sigma$}

Recall from Section \ref{plumbings} that a 4-manifold with boundary can be constructed by plumbing disc bundles over $S^2$ according to a graph $G$. In this instance our $X = X(G)$ uses the same graph as in Figure \ref{fig:goodplumbing}. As $X$ is simply connected, we can identify $H^2(X) = \Hom(H_1(X),\mathbb{Z})$, and so $H^2(X)$ is the $\mathbb{Z}$-module spanned by the $\Hom$-duals $[S_v]^*$. By mapping $H_2(X)$ to $H^2(X,\partial X)$ via Poincar\'e duality, we have the following commutative diagram:
\begin{equation}
\label{exact}
\begin{CD}
& & & &  \Spinc(X) @>>> \Spinc(\Sigma) &\\
& & & & @V{\mathfrak{c}_1}VV @V{\mathfrak{c}_1}VV &\\
0 @>>> H_2(X) @>>> H^2(X) @>>> H^2(\Sigma) @>>> 0\\
& & @VVV @VVV @VVV &\\
0 @>>> \mathbb{Z}^{b_2(X)} @>{Q}>> \mathbb{Z}^{b_2(X)} @>{\alpha}>> \coker Q @>>> 0
\end{CD}
\end{equation}
The vertical maps between the lower two rows are isomorphisms and we use them to identify each of their domains and codomains. From the middle row, it is now clear that $\ker \alpha$ is spanned by those $K$ which are $\mathbb{Z}$-linear combinations of the rows of $Q$.

To see how the $\Spinc$-structures on $X$ and $\Sigma$ fit into this picture, define the set of characteristic covectors for $G$, denoted $\Char(G)$, to be those $K \in H^2(X)$ such that
$$\left<K,[S_v]\right>\equiv \left<[S_v], [S_v]\right> \mod 2 \text{ for all } v \in V(G).$$
Now, it is well known that the $\Spinc$-structures on $X$ correspond precisely with $\Char(G)$ via $\mathfrak{c}_1$; similarly, the $\Spinc$-structures on $\Sigma$ are in bijection with $2H^2(\Sigma)$ (also via $\mathfrak{c}_1$). Since $H^2(\Sigma)$ is of odd order, the $\Spinc$-structures are therefore in bijection with $H^2(\Sigma)$ and hence also with $\coker(Q)$.

What we want, then, is a good set of representatives for $\coker(Q)$, since these will represent the $\Spinc$-structures on $\Sigma$. We write $\mathfrak{t}(K)$ for the $\Spinc$-structure on $\Sigma$ determined by the equivalence class $[K]$ in $\coker (Q)$ determined by $K \in \Char(G)$. We observe that if $\mathfrak{s}_1,\mathfrak{s}_2 \in \Spinc(X)$ restrict to the same $\Spinc$-structure on $\Sigma$, then their corresponding covectors (respectively $K_1,K_2$) are congruent modulo $2H_2(X)$. In other words, $K_1 \equiv K_2 \mod Q$, since $H^2(\Sigma)$ is of odd order, or $(K_1-K_2)Q^{-1} \in \mathbb{Z}^{b_2(X)}$.

The results that Ozsv\'ath and Szab\'o give in \cite{OSPlumbed} state that, assuming $Q$ is negative-definite and that there is at most one overweight vertex in the graph, the correction term $d(Y,\mathfrak{t})$ is
\begin{equation}
\label{compute}
d(Y,\mathfrak{t}) = \frac{1}{4}\left(\max_{K : \mathfrak{t}(K) = \mathfrak{t}} \left<K,K\right> + \abs{G}\right).
\end{equation}
A vertex $v$ is \emph{overweight} if $w(v) > -d(v)$, where $d(v)$ is the degree of $v$.

We will write all elements $K \in H^2(X)$ according to their evaluations on all $v \in V(G)$ (that is, in the $\Hom$-dual basis). Written thus, the square in \eqref{compute} is $KQ^{-1}K^t$. Incidentally, \eqref{compute} also shows that $X$ is sharp.

Having set this all up, the actual algorithm for finding the maximisers $K$ for \eqref{compute} is as follows. Consider all $K \in\Char(G)$ that satisfy
\begin{equation}
\label{startpath}
w(v) + 2 \leq \left<K,[S_v]\right> \leq -w(v).
\end{equation}
Set $K_0 := K$. Then, should one exist, choose a $v_{i+1}\in V(G)$ such that
$$\left<K_i,[S_{v_{i+1}}]\right> = -w(v_{i+1}),$$
and set $K_{i+1} := K_i + 2\PD[v_{i+1}]$ (which we will refer to as \emph{pushing down} the value of $K_i$ on $v_{i+1}$). By $\PD[v]$ we mean the image of $\PD[S_v]$ in $H^2(X)$ using \eqref{exact}. Pushing down then amounts to adding two copies of the corresponding row of $Q$.

After continuing in this fashion, terminate at some covector $K_n$ when one of two things happens. Either
$$w(v) \leq \left<K_n,v\right> \leq -w(v) -2 \text{ for all } v \in V(G),$$
in which case we say the path $(K_0,K_1,\dots,K_n)$ is \emph{maximising}, or
$$\left<K_n,v\right> > -w(v) \text{ for some } v\in V(G),$$
in which case the path is \emph{non-maximising}. Ozsv\'ath and Szab\'o show (in Proposition 3.2 of \cite{OSPlumbed}) that the maximisers required for computing correction terms can be taken from a set of characteristic covectors $\Charm(G)$ satisfying \eqref{startpath} with the additional property that they initiate a maximising path.

To apply this to our pretzel, consider the plumbing in Figure \ref{fig:goodplumbing}. Note that $v_k$ is the central (3-valent) vertex. With the labelling specified there, our intersection form $Q$ has matrix
$$Q = \begin{pmatrix}-2 &1 &   &  & & & &\\
1 & -2 &1  & & & & & \\
  &1  &-2  & & & & & \\
 & & &\ddots & & & & \\
 & & &     &-2 &1 & &\\
 & & &      &1 &-2 &1 &1\\
 & & &      & &1 &-2 & \\
 & & &      & &1 &   &-k
\end{pmatrix},$$
and as there is only one overweight vertex (the central one), the above algorithm is applicable. We present the result of it below.

\begin{prop}
The following characteristic covectors initiate maximising paths:
\begin{enumerate}
\item $(0,0, \dots, 0,2,0, \dots, 0,j)$, where the $2$ is in the $i^{\text{th}}$ place, and $j\in \mathbb{Z}$ is odd and $2-k \leq j \leq k-4$;
\item $(2,0, \dots, 0,0,k-2)$ and $(0,\dots,0,2,k-2)$; and
\item $(0,\dots,0,j)$ where $j$ is an odd integer satisfying $2-k \leq j \leq k$.
\end{enumerate}
Moreover, there are no other vectors that initiate full paths.
\end{prop}
\begin{proof}
Let $K \in \Char(G)$ satisfy \eqref{startpath}. We show that if $K$ satisfies either of two conditions below then it must initiate a non-maximising path.

First suppose that there are two $v \in V(G)$ such that $\left<K,[S_v]\right> = -w(v)$. Then, pushing down at $v_{k+2}$ if necessary, we have a substring of $K$ that looks like $(2,0,\dots, 0, 2)$. On pushing down the 2's and then iterating with the 2's created within this substring, we eventually obtain a value 4 in the substring and thus $K$ initiates a non-maximising path.

We now consider $K = (0,\dots,0,2_{(i)},0,\dots, k-2)$ for $i = 2,\dots, k$. Push down the 2 here to create 2's on either side. Keep pushing these newly created 2's down in either direction; the result is $(-2,0,\dots,0,2_{(i)},0,\dots,-2,k)$. On repeating this procedure, we end up with $k+2$ in the final co-ordinate. As this is too large, $K$ initiates a non-maximising path.

The remaining $K$, then, are precisely those listed above. Since there are $(k+1)(k-2) = k^2 -k - 2$ such vectors of the first kind, 2 of the second, and $k$ of the last, we have $k^2$ in total. That being the order of $H^2(\Sigma)$, these must initiate maximising paths and enumerate the different $\Spinc$-structures on $\Sigma$.
\end{proof}

We give the maximisers $\Charm(G)$ the following names:
$$K^1_{i,j} := (0,\dots,2_{(i)},\dots, 0,j) \text{ for $j$ odd, $2-k\leq j \leq k- 4$}$$
$$K^2_1 :=(2,0,\dots, k-2) \qquad \qquad \qquad K^2_2:=(0,\dots,0,2,k-2)$$
$$K^3_j :=(0,\dots,0,j) \text{ for $j$ odd, $2-k\leq j \leq k$}$$
It will sometimes be useful to allow $K^1_{i,j}$ to define the same type of covector as above, but with a value for $j$ outside the range and parity specified. In this case we emphasise that it does not represent a maximiser useful for calculating the corresponding correction term. Following this trend, it is also sometimes useful to set
$$K^1_{0,j} := K^3_j.$$

To compute the correction terms, then, we need $Q^{-1}$. This calculation is surprisingly tractable so we present the result directly: $Q^{-1} = \tfrac{1}{k^2}(c_{ij})$, where
$$c_{ij} = \begin{cases} -i(k^2-jk+2j) & i\leq j \leq k-1\\
-2jk& i = k, j\leq k\\
-jk & i = k+1, j \leq k\\
-k^2 & i = j = k+1\\ 
-2j & i =k+2, j \leq k\\
-k & i = k+2, j = k+1\\
-(k+2) & i = j = k+2\\
c_{ji} & \text{all } i,j
\end{cases}.$$

This in turn permits an explicit calculus of the squares below:
$$(K^1_{i,j})^2 =\begin{cases} - \tfrac{1}{k^2} (4i(k^2-ik+2i) + (k+2)j^2+8ij) & \text{for }i = 0,\dots,k\\
- \tfrac{1}{k^2}(4k^2 +(k+2)j^2 +4kj) &\text{for }i = k+1\end{cases}$$
$$(K^2_i)^2 =\begin{cases}- (k+2) &\text{for }i = 1\\
- \tfrac{1}{k^2}(k^3+6k^2-12k+8) &\text{for }i = 2\end{cases}$$
The computation of $d(\Sigma,\mathfrak{t}(K))$ is then trivial. In what follows we write $d(\cdot, K)$ in place of $d(\cdot, \mathfrak{t}(K))$ as the meaning is clear.

\subsection{Correction Terms for the Lens Space $L$}

We need to repeat this procedure for the corresponding lens space, $L(D,2)$, which has a plumbing given by the tree $H$ on two vertices, weighted $-n$ and $-2$ (recall $k^2 = 2n-1 = 4s+1$). It therefore has intersection form given by
$$R_n = \begin{pmatrix}-n &1\\1&-2\end{pmatrix},$$
and in this case the inverse is trivially
$$R^{-1}_n = -\frac{1}{k^2} \begin{pmatrix}2 &1\\1 &n\end{pmatrix}.$$

What remains is then to establish the labelling of the $\Spin$-structures on $L$ as they are required for Theorem \ref{OSunknot} and compute their correction terms. This is in fact already done by Ozsv\'ath and Szab\'o in \cite{OSPlumbed}.

\begin{lem}[Ozsv\'ath and Szab\'o]
The lens space $L(D,2)$ has characteristic covectors given by the map $\psi : \mathbb{Z}/(D) \longrightarrow \coker(R_n)$, defined below.
\begin{equation}
\label{lens}
\psi(i) = \begin{cases}(2i-1,2) &0 \leq i\leq s\\
(2i-4s-1,0) &s+1 \leq i \leq 3s+1\\
(2i-8s-3,2) &3s+2 \leq i \leq 4s \end{cases}.
\end{equation}
\end{lem}

To compute the correction terms, we write
$$d(L,\psi(i)) = \frac{-\psi(i)\begin{pmatrix}2&1\\1&n\end{pmatrix}\psi(i)^t + 2k^2}{4k^2},$$
or more explicitly,
\begin{align*}
d(L,\psi(i)) = \begin{cases}
-\tfrac{1}{k^2}(2i^2) &0 \leq i\leq s\\
-\tfrac{1}{2k^2}((2i-k^2)^2-k^2) & s+1 \leq i \leq 3s+1\\
-\tfrac{1}{k^2}(2(k^2-i)^2) &3s+2 \leq i \leq 4s 
\end{cases}.
\end{align*}

\subsection{First Application of Theorem \ref{OSunknot}}

To compare our correction terms for $\Sigma$ and $L$ we will need the isomorphism $\varphi : \mathbb{Z}/(D) \rightarrow \coker(Q)$ implicit in Theorem \ref{OSunknot}. Since $\varphi$ was only implicit in our statement of the theorem, let us be clear what we are doing. Suppose we can construct a particular isomorphism $\phi : \mathbb{Z}/(D) \rightarrow \coker(Q)$. Then we have two labellings of $\Spinc$-structures: one, $\psi$, for $\Spinc(L)$, the other, $\phi$, for $\Spinc(\Sigma)$. Theorem \ref{OSunknot} then tells us that there exists a labelling $\varphi$ for $\Spinc(\Sigma)$ such that
\begin{equation}
\label{obstruction}
d(\Sigma,\varphi(i)) - d(\Sigma,\varphi(2s-i)) = d(L,\psi(i)) - d(L,\psi(2s-i))
\end{equation}
for $i = 0,1, \dots, s$. We want to show, therefore, that for $k \geq 5$ such a $\varphi$ cannot exist by way of contradiction: if $\varphi$ \emph{did} exist, then we could precompose our $\phi : \mathbb{Z}/(D) \rightarrow \coker(Q)$ with some automorphism of $\mathbb{Z}/(D)$ such that the equations \eqref{obstruction} are satisfied. This automorphism must be multiplication by some $\ell$ coprime to $k^2$. That is, there must exist some $\ell$ such that
\begin{equation}
\label{obs2}
d(\Sigma,\phi(i\ell)) - d(\Sigma,\phi(2s\ell - i\ell)) = d(L,\psi(i))-d(L,\psi(2s-i)).
\end{equation}
To prove that $u(K) > 1$ when $k \geq 5$, we claim that no such $\ell$ exists, and it is in this direction that we proceed over the course of the following pages. First, however, we must specify a $\phi$ by specifying $\phi(1)$. We must also compute $\phi(0)$ to check the hypotheses of Theorem \ref{OSunknot}, but in this we have no choice.

As we saw in \eqref{exact}, the kernel of $\alpha$ is generated by $\PD[v]$. Thus, on observing that
\begin{equation}
\label{zeroelement}
K^2_1 = -2\sum_{i=1}^k \PD[v_i] - \PD[v_{k+1}] - \PD[v_{k+2}],
\end{equation}
we see that $K^2_1$ is the zero element of $\coker (Q)$.

Now to find a unit. This time we need to find the $K$ such that $m[K] = [K^2_1]$ if and only $k^2 \vert m$. Equivalently, the $K$ such that $(mK-K^2_1)Q^{-1} \in \mathbb{Z}^{k+2}$ if and only $k^2 \vert m$. Setting $K = K^1_{1,-1}$, we have
$$mK-K^2_1 = (2(m-1),0,\dots, 0, -m-(k-2)),$$
and on computing the $(k+2)$-th co-ordinate of $(mK-K^2_1)Q^{-1}$ we find
$$((mK-K^2_1)Q^{-1})_{k+2} = \tfrac{1}{k^2}(k^2-m(k-2)) \in \mathbb{Z}.$$
It follows that $k^2\vert m$ since $k^2$ and $k-2$ are coprime, $k$ being odd. Thus $K^1_{1,-1}$ is a unit.

Our choice of $\phi$, which we now fix, is then specified by
\begin{align*}
\phi(0) = K^2_1 && \phi(1) = K^1_{1,-1}.
\end{align*}
At this point, we are able to check that \eqref{zero} in Theorem \ref{OSunknot} is satisfied. Using the same sort of methods, we can check the same hypothesis for Theorem \ref{greene} for our pretzels $P(k,-k-2,2m)$, as promised in Section \ref{correctsharp}.

\begin{lem}
\label{check}
If $K = P(k,-k,2),P(k,-k-2,2m)$, then $d(\Sigma(K),0) = -d(L(\det K,2),0)$. In the first case, both correction terms vanish.
\end{lem}
\begin{proof}
We do $P(k,-k,2)$ first. By \eqref{compute} we have
$$d(\Sigma(K),0) = \frac{(K^2_1)Q^{-1}(K^2_1)^t + (k+2)}{4}.$$
Since $\PD[v_i]$ is the $i$-th row of $Q$, it follows that if $K^\prime = \sum_{i=1}^{k+2} k_i \PD[v_i]$ then $K^\prime Q^{-1} = \sum_{i=1}^k k_i e_i$, where $e_i$ is the $i$-th standard basis vector for $\mathbb{Z}^{k+2}$. Thus the square in the above correction term is nothing more than $\sum_{i = 1}^k k_i \left<K^\prime,[S_{v_i}]\right>$, and on computing this for $K^\prime = K^2_1$ using \eqref{zeroelement} we find that $d(\Sigma(K),0) = 0$. Since $\det K \equiv 1 \mod 4$, it is also true that $d(L(\det K,2),0) = 0$ (calculated similarly), and we are done.

Note that as we had already computed $(K^2_1)^2$ we could have done this proof immediately. However, the method just presented allows us to generalise to $P(k,-k-2,2m)$, whose double branched cover is also an $L$-space. There we replace \eqref{zeroelement} with
\begin{align*}
K &= -\sum_{i=1}^m 2i \PD[v_i] - 2m\sum_{i=m+1}^k\PD[v_i] - \sum_{i=k+1}^{k+2m-1} (k+2m-i) \PD[v_i] - \PD[v_{k+2m}]\\
&=(0,\dots,0, 2_{(m)}, 0,\dots,0,k+2-2m)
\end{align*}
if $k \geq m$, or
\begin{align*}
K &=-\sum_{i=1}^k 2i\PD[v_i] - \sum_{i=k+1}^{m} (k+i)\PD[v_i] - \sum_{i=m+1}^{k+2m-1} (k+2m-i)\PD[v_i] - \PD[v_{k+2m}]\\
&= (0,\dots,0,2_{(m)},0,\dots,0,-k+2)
\end{align*}
if $k \leq m$. One can check that these initiate maximising paths, and as $\Sigma(K)$ is an $L$-space these must be unique representatives of the zero $\Spinc$-structure. This tells us that $d(\Sigma(K),0) = -\half$. Similarly, as $\det K \equiv 3 \mod 4$, we have $d(L(\det K,2),0) = \half$, and we are done.
\end{proof}

Having now checked all the hypotheses and set ourselves up for Theorem \ref{OSunknot}, we now make our first applications of it.

\begin{prop}
\label{congruence}
Suppose $P(k,-k,2)$ has unknotting number one. Then there exists an $\ell$ coprime to $k^2$ such that
$$\ell^2(3k-2) \equiv -8 \mod k^2.$$
Equivalently,
\begin{equation}
\label{cong}
\ell^2 \equiv 6k+4 \mod k^2
\end{equation}
\end{prop}
\begin{proof}
We observe that if $\phi(i) = [K]$, then $\phi(i\ell) = \ell\phi(i) = \ell[K] = [\ell K]$. Thus $\phi(\ell i)Q^{-1} \equiv \ell \phi(i)Q^{-1} \mod \mathbb{Z}$, and so $\phi(\ell i)^2 \equiv \ell^2 \phi(i)^2 \mod \mathbb{Z}$. Thus,
\begin{equation}
\label{equiv}
d(\Sigma,\phi(i\ell)) - d(\Sigma,\phi(2s\ell - i\ell)) \equiv \ell^2 (d(\Sigma,\phi(i)) - d(\Sigma,\phi(2s-i))) \mod \mathbb{Z}.
\end{equation}

Now use our previous calculation of $L$'s correction terms:
$$d(L,\psi(0)) - d(L,\psi(2s)) = -\tfrac{1}{2k^2}(k^2-1).$$
It is also a routine matter of calculation to find
$$\phi(2s)= \begin{cases}K^1_{k+1,-\half(k+1)} &k \equiv 1 \mod 4\\K^3_{\half(k-1)} & k \equiv 3 \mod 4\end{cases},$$
from which we deduce that
$$d(\Sigma,\phi(0))-d(\Sigma,\phi(2s)) = \begin{cases}\tfrac{1}{16k^2}(5k^3-3k+2) &k\equiv 1 \mod 4\\ \tfrac{1}{16k^2}(-3k^3+4k^2-3k+2) &k\equiv 3 \mod 4\end{cases}.$$

Applying \eqref{equiv} to \eqref{obs2} in the case when $i = 0$ and substituting in the above calculations, we find that we must have
$$-8(k^2-1) \equiv \begin{cases} \ell^2 (5k^3-3k+2) \mod k^2 & \text{if } k \equiv 1 \mod 4\\
\ell^2 (-3k^3+4k^2-3k+2) \mod k^2 &\text{if } k \equiv 3 \mod 4\end{cases},$$
which transforms into the equivalent statement \eqref{cong} after a simple rearrangement (to make $\ell^2$ the subject).
\end{proof}

As a remark, if we look at \eqref{obs2} modulo $\mathbb{Z}$ for any other value of $i$ we recover the same congruence. Therefore no further information is to be gained along these lines. However, \eqref{cong} by itself cuts down the number of possible $\ell$ considerably. In the case that $k$ is a prime power, for instance, it determines $k$ up to sign (see next section).

\subsection{Precise Applications of Theorem \ref{OSunknot}}

The rest of the proof that $u(K) > 1$ for $k \geq 5$ follows the following line of reasoning. We show that we cannot satisfy \eqref{obs2} with an $\ell$ satisfying \eqref{cong} for $i = 0$ and $r$ simultaneously, where $r$ is the residue of $\ell$ modulo $k$. This requires us to do the following:
\begin{enumerate}
\item Pinpoint the values of $\phi (2s\ell)$, $\phi(r \ell)$, $\phi(2s \ell - r\ell)$, and compute their squares;
\item Compute the differences
$$Z(i) := d(\Sigma,\phi(i\ell))-d(\Sigma,\phi(2s\ell - i\ell)) - d(L,\psi(i)) + d(L,\psi(2s-i))$$
for $i = 0,r$;
\item Obtain a good reason why $Z(0)$ and $Z(r)$ cannot simultaneously be zero for $k \geq 5$.
\end{enumerate}

For the reader who does not like results plucked out of thin air, we can provide some comments on the combinatorics involved in the group structure on $\Charm(G)$. The following formulae are the tools used to compute the values of $\phi$ called for in the first step above.

\begin{lem}
\label{exchange}
We have the following equivalences (in $H^2(X)$):
\begin{align*}
\textbf{(A): } & K^3_{J+kB} \sim K^1_{-B,J+2B} & \textbf{(B): } &K^1_{I,J} \sim K^1_{I+1,J+k-2} & \textbf{(C): } &K^1_{I,J} \sim K^1_{I,J+k^2}.
\end{align*}
where $B \leq 0$ and $J$ are arbitrary integers.
\end{lem}
\begin{proof}
This is an easy calculation: simply verify that $(K-K^\prime)Q^{-1} \in \mathbb{Z}^{k+2}$, for the above $K,K^\prime$.
\end{proof}

As mentioned before, if $k$ is a prime power then there is an essentially unique choice of $\ell$. The situation becomes much more complicated if $k$ has several different prime factors; to deal with this complexity, we introduce some auxiliary notation.

\begin{prop}
Let $\ell = ak + r$, where $0 \leq a <k$ and $0 < r < k$. Then we can choose $r$ even and set $r^2 = Ak +4$, where
\begin{equation}
\label{magic}
A +2ar \equiv 6 \mod k,
\end{equation}
and $0 \leq r - A < \tfrac{k}{4}+1$.
\end{prop}
\begin{proof}
Since $\pm \ell$ have the same effect on the correction terms, and $k$ is odd, one of $\pm \ell$ will have even $r$ and we make this choice. Notice that as $\ell$ is coprime to $k$, we cannot have $r = 0$.

From \eqref{cong}, $\ell^2 \equiv 6k + 4 \mod k^2$, but also $\ell^2 \equiv 2ark + r^2 \mod k^2$, and substituting gives the desired congruence \eqref{magic}.

For the inequality, we have $r - A = r - \tfrac{r^2-4}{k}$. By considering this quadratic in the range from $0$ to $k$, we find it is always positive, maximises when $r = \tfrac{k}{2}$, and has maximum $\tfrac{k}{4} + \tfrac{4}{k}$. Since $r-A$ is an integer, and as $k\geq 5$, the upper bound follows.
\end{proof}

\begin{prop}
\label{primep}
In the case that $k$ is a prime power, then $r = 2$, $A = 0$, and $a = \tfrac{k+3}{2}$.
\end{prop}
\begin{proof}
This is a direct calculation using \eqref{cong}, and the observation that when $k$ is prime power, square roots modulo $k$ are unique up to sign.
\end{proof}

To carry out our programme we now have to branch out into several different cases. Since the condition that $r$ is even implies nothing about $a$ and the parity of $a$ becomes important in what follows, we divide our proof into sthe following two sections according to whether $a$ is even or odd. However, in Step One of our recipe, one value of $\phi$ turns out to be independent of $a$.

\begin{prop}
For $k \geq 5$,
$$\phi(r\ell) = -K^1_{\tfrac{A}{2},k-4-A}.$$
\end{prop}
\begin{proof}
By direct verification. Check that $(-r\ell K^1_{1,-1} - K^1_{\tfrac{A}{2},k-4-A})Q^{-1} \in \mathbb{Z}^{k+2}$, which is easy.
\end{proof}

One final notational remark. Strictly speaking, we want to compute $d(\Sigma,i)$, but as there is the conjugation symmetry $d(Y,i) = d(Y,-i)$, sometimes we will in fact compute $\phi(-i)$ instead of $\phi(i)$. We write $\phi(i) = -K$ to mean $\phi(-i) = K$ by an abuse of notation to streamline our statements.

\subsubsection{The Case $a$ Even}

According to Step One we must now compute the values of $\phi(2s\ell),\phi(2s\ell-r\ell)$. This is done in the following two propositions. For the interested reader, these calculations were performed originally by assuming that $K$ had the form $K^3_j$, and then applying Lemma \ref{exchange} until the subscripts fitted their required conditions.

\begin{prop}[$r \equiv 2 \mod 4$]
If $r \equiv 2 \mod 4$ and $a$ is even, then we define parameters $B:= 1 + \tfrac{r}{2}-\tfrac{a}{2}-\tfrac{A}{2} \in (-\tfrac{k}{2},\tfrac{k}{2})$ and $J:= -\tfrac{r}{2}{-4} < 0$. These give
$$\phi(2s\ell-r\ell) = \begin{cases}
K^1_{-B,J+2B} &\text{if } B\leq 0, J+2B > -k\\
K^1_{2-B,J+2B+2k-4} &\text{if } B\leq 0, J+2B \leq -k\\
-K^1_{B,-J-2B} &\text{if }B \geq 0
\end{cases},$$
and also
$$\phi(2s\ell) = \begin{cases}K^1_{\tfrac{a-r}{2},\tfrac{r}{2}-a} &\text{if }a \geq r\\
-K^1_{\tfrac{r-a}{2},a-\tfrac{r}{2}} &\text{if }a \leq r
\end{cases}.$$
\end{prop}

\begin{prop}[$r \equiv 0 \mod 4$]
If, on the other hand, $r \equiv 0 \mod 4$, then we instead define $B := \tfrac{r}{2}-\tfrac{a}{2} - \tfrac{A}{2} \in (-\tfrac{k}{2},\tfrac{k}{2})$ and $J:=k-\tfrac{r}{2}-4>0$, giving
$$\phi(2s\ell-r\ell) = \begin{cases}
K^1_{-B,J+2B}&\text{if }B \leq 0\\
-K^1_{B,-J-2B} &\text{if }B \geq 0, J+2B < k\\
-K^1_{B+2,-J-2B+2k-4} &\text{if }B \geq 0, J+2B \geq k\\
\end{cases},$$
and also
$$\phi(2s\ell) = \begin{cases}
K^1_{\tfrac{a-r+2}{2},\tfrac{r}{2}-a+k-2} &\text{if }a \geq r-2\\
-K^1_{\tfrac{r-a-2}{2},a-\tfrac{r}{2}-k+2} &\text{if } \tfrac{r}{2} \leq a\leq r-2\\
-K^1_{\tfrac{r-a+2}{2},a-\tfrac{r}{2}+k-2} &\text{if } a < \tfrac{r}{2}
\end{cases}.$$
\end{prop}
\begin{proof}[Proof (of both propositions)]
This is a straightforward verification. To perform it, one need only check that $(mK^1_{1,-1} - K)Q^{-1} \in \mathbb{Z}^{k+2}$ for the right choices of $m$ and $K$ from the above. In doing so, one must use the congruence \eqref{magic} to guarantee the result. The numerous cases occur to fit the various restraints imposed on $i,j$ in $K^1_{i,j}$; the fact that $\tfrac{r}{2}$ is odd makes a difference is because of the fact that $j$ must be odd.
\end{proof}

This completes Step One. The next step is to compute $Z(i)$ for $i = 0,r$. As this is straightforward, if tedious, we present the result immediately. In the tables below, the ``case'' label will become relevant later. First for $i = r$:

\begin{center}
\begin{tabular}{c|c|c||c}
Case &$r \mod 4$ &Conditions &$16k^2Z(r)$\\
\hline
\hline
$A$ &$2$ & $B\leq 0$ &$(4kr+8k^2)A+(2-3k)r^2 + ((4a-24)k-8k^2)r$\\
& &$J+2B > -k$ & $\phantom{SPACE}-4k^3 + (8a+16)k^2 + 32ak -8$\\
\hline
$B$ &$2$ & $B \leq 0$ &$(4kr-8k^2)A+(2-3k)r^2 + (4a-24)kr$\\
& &$J+2B \leq -k$ & $\phantom{SPACE}+12k^3 + (-8a-16)k^2 + 32ak -8$\\
\hline
$C$ &$2$ & $B \geq 0$ &$(4kr-8k^2)A+(2-3k)r^2 + ((4a-24)k+8k^2)r$\\
& & & $\phantom{SPACE}-4k^3 + (-8a+48)k^2 + 32ak -8$\\
\hline
\hline
$D$ &$0$ & $B \leq 0$ &$4Akr +(2-3k)r^2 + ((4a-24)k -4k^2)r$\\
& & & $\phantom{SPACE}+8k^2 + 32ak -8$\\
\hline
$E$ &$0$ & $B \geq 0$ &$(4kr-16k^2)A + (2-3k)r^2 + ((4a-24)k+12k^2)r$\\
& &$J + 2B < k$ &$\phantom{SPACE}+(-16a+8)k^2+ 32ak -8$ \\
\hline
$F$ &$0$ & $B \geq 0$ &$4Akr + (2-3k)r^2 + ((4a-24)k +4k^2)r $\\
& &$J + 2B \geq k$ &$\phantom{SPACE}+72k^2 + 32ak -8$
\end{tabular}
\end{center}

And now for $i =0$:

\begin{center}
\begin{tabular}{c|c|c||c}
Case &$r \mod 4$ &Conditions &$16k^2Z(0)$\\
\hline
\hline
$1$ &$2$ &$a \geq r$ &$(2-3k)r^2 + (4ak-8k^2)r -4k^3 + 8ak^2 - 8$\\
$2$ &$2$ &$a \leq r$ &$(2-3k)r^2 + (4ak+8k^2)r -4k^3 - 8ak^2 - 8$\\
\hline
$3$ &$0$ &$a\geq r-2$ &$(2-3k)r^2 +(4ak-4k^2)r +8k^2 -8$\\
$4$ &$0$ &$\tfrac{r}{2} \leq a\leq r-2$ &$(2-3k)r^2 +(4ak+12k^2)r-(16a+24)k^2 -8$\\
$5$ &$0$ &$a < \tfrac{r}{2}$ &$(2-3k)r^2 + (4ak+4k^2)r + 8k^2 - 8$
\end{tabular}
\end{center}

This completes Step Two. We remark that since all the above entries must be zero, we can manipulate them and divide out any resulting common factors (such as $4k$) without sacrificing equality with zero. These reduced versions are what we will often use.

\begin{prop}
If $a$ is even, then no $\ell$ exists which ensures that $Z(r) = Z(0) = 0$.
\end{prop}
\begin{proof}
The idea is to show that none of the $Z(r) = 0$ equations in case $\alpha$ is compatible with any of the $Z(0) = 0$ equations in case $\beta$ (for appropriate choices of $\alpha$ and $\beta$). If both the $\alpha$ and $\beta$ equations are satisfied, then we should have
$$Z(r)\pm Z(0) =0.$$
Thus we must compare cases $\alpha = A,B,C$ with cases $\beta = 1,2$ (six combinations), as well as cases $\alpha = D,E,F$ with cases $\beta = 3,4,5$ (nine more combinations). In each case, both of the new equations generally involve $A$, $a$, and $r$, so obtaining contradictions can be difficult. The following method is useful in a large number of cases.
\begin{enumerate}
\item Cancel sufficient common factors from all the terms;
\item Substitute $A = \tfrac{r^2-4}{k}$;
\item Solve the $Z(r)+Z(0) = 0$ equation for $a$ (linear) and substitute it into the $Z(r)-Z(0)=0$ equation, taking care to observe that the coefficient of $a$ in $Z(r)+Z(0)=0$ is non-zero (so there are no ``divide by zero'' issues). This gives a new equation $f_{\alpha,\beta}(r) = 0$ to be satisfied;
\item Find an argument to prove that the function $f_{\alpha,\beta}$ is positive or negative over the range $2,4 \leq r < k$. The choice of 2 or 4 depends on the minimum value of $r$ allowed by $\alpha$ and $\beta$.
\item Hence, conclude that $\alpha,\beta$ are not compatible.
\end{enumerate}

We illustrate the procedure once, then just summarise the relevant $f_{\alpha,\beta}$. Take $\alpha = A$ and $\beta = 1$. Cancelling terms, we obtain:
\begin{align*}
Z(r)+Z(0) &= 0 = (2r+k+2)(r^2-4) - (8k+12)kr - 4k^3 + 8k^2  -12k + 4ak(r +2k+ 4)\\
Z(r)-Z(0) &= 0 = (r+2k)(r^2-4) -6rk + 4k^2 + 8ak.
\end{align*}
Observe that the coefficient of $a$ in the first equation is non-zero, so solving for $a$ and substituting directly into the second, we find that
$$f_{A,1}(r) = r^4 + 4kr^3 + (4k^2-8)r^2+(8k^2-16k)r + 16k^3-16k^2+16 = 0.$$
Since $k \geq 5$, the coefficients of $r$ are all positive, whence $f_{A,1}(r) > 0$ on $0 < r < k$, giving the contradiction we require.

In a similar vein, we now summarise the other data in the following table.

\begin{center}
\begin{tabular}{c|c}
$\alpha,\beta$ &$f_{\alpha,\beta}(r)$\\
\hline
\hline
$A,1$ &$r^4 + 4kr^3 + (4k^2-8)r^2+(8k^2-16k)r + 16k^3-16k^2+16$\\
$B,1$ &$r^4 -(5k^2-2k+8)r^2 + 4k^4 + 16k^2 - 8k +16$\\
$C,1$ &$r^4 -(3k^2-2k+8)r^2 + 16k^2r - 4k^4+32k^3+16k^2-8k+16$\\
\hline
$A,2$ &$r^4-(5k^2+2k+8)r^2 + 4k^4 + 16k^2+8k+16$\\
$B,2$ &$r^4 - 4kr^3 + (2k^2-8)r^2 + (8k^3-8k^2+16k)r -8k^4 + 16k^3-16k^2+16$\\
$C,2$ &$r^4 - 4kr^3 + (4k^2-8)r^2 + (8k^2+16k)r-16k^3-16k^2+16$\\
\hline
\hline
$D,3$ &$r^4-8r^2+8k^2r-16k^2+16$\\
$E,3$ &$r^4-4kr^3 + (k^2+2k-8)r^2 - (4k^3-8k^2-16k)r +8k^3-16k^2-8k+16$\\
$F,3$ &$r^4 + (2k^2-8)r^2 + 24k^2r - 16k^2+16$\\
\hline
$D,4$ &$r^4 -4kr^3 -(k^2+2k+8)r^2 +(4k^3+8k^2+16k)r-8k^3+48k^2+8k+16$\\
$E,4$ &$r^4 -8kr^3 +(16k^2-8)r^2 + (8k^2+32k)r -32k^3-16k^2+16$\\
$F,4$ &$r^4 -4kr^3 +(k^2-2k-8)r^2 - (4k^3-24k^2-16k)r -72k^3+48k^2+8k+16$\\
\hline
$D,5$ &$r^4 -(2k^2+8)r^2 - 8k^2r -16k^2 + 16$\\
$E,5$ &$r^4 - 4kr^3 - (k^2-2k+8)r^2 + (4k^3-8k^2+16k)r + 8k^3 -16k^2 - 8k + 16$\\
$F,5$ &$r^4 - 8r^2 + 8k^2r - 16k^2 + 16$\\
\end{tabular}
\end{center}

We attack these cases case by case.
\begin{description}
\item[A1] Already done.
\item[B1, A2] In both situations, $f_{\alpha,\beta} = r^4 - Nr^2 + M$. The turning points of this quartic occur when $r=0$ or $r^2 = \tfrac{N}{2}$, so provided that $\tfrac{N}{2} \geq k^2$, we know that $f_{\alpha,\beta}$ is decreasing on $0 < r < k$. As this happens to be true, and
$$f_{\alpha,\beta}(k-1) = \begin{cases}8k^3+5k^2+6k+9 &\text{if }\alpha = B, \beta = 1\\
4k^3+13k^2+18k+9 &\text{if }\alpha = A, \beta = 2\end{cases},$$
we see that $f_{\alpha,\beta}(r)> 0$ on $0 < r < k$, which is our contradiction.
\item[C1] The function is not obviously useful, but we know $1 + \tfrac{r}{2}-\tfrac{a}{2}-\tfrac{A}{2} \geq 0$ (by case C) and $a\geq r$ (by case 1), whence we are forced to conclude that $A = 0$. However, then $r = 2$ and $a = \tfrac{k+3}{2}$ by direct computation, and the condition from case C fails. Contradiction.
\item[B2] We aim to show that $f(r): = f_{B,2}(r)<0$ on $0 < r < k$ and for $k \geq 7$. Indeed, compute the derivatives:
\begin{align*}
\frac{df}{dr}(r) &= 4r^3-12kr^2+(4k^2-16)r + (8k^3-8k^2+16k)\\
\frac{d^2f}{dr^2}(r) &=12r^2-24kr+(4k^2-16)\\
\frac{d^3f}{dr^3}(r) &=24r - 24k.
\end{align*}
As we can see, $\tfrac{d^3f}{dr^3}(r)<0$, whence $\tfrac{d^2f}{dr^2}$ is decreasing. Observing that $\tfrac{d^2f}{dr^2}(0) = 4k^2 - 16 > 0$ while $\tfrac{d^2f}{dr^2}(k) = -8k^2-16 < 0$, we know there is precisely one zero in the range $0 < r < k$. Hence, $\tfrac{df}{dr}$ has one turning point, and it is a maximum by the negativity of $\tfrac{d^3f}{dr^3}$. Checking at both extremes of the range again finds that $\tfrac{df}{dr}(r) > 0$, and so $f$ is increasing. However,
$$f(k) = -k^4+8k^3-8k^2+16,$$
which is negative for $k \geq 7$, and so $f_{B,2}(r) = f(r) < 0$ on the range prescribed. If $k = 5$, then observe that $r = 2$, and direct computation finds $f_{B,2}(2) < 0$.
\item[C2] We play around with the $Z(r)-Z(0)=0$ equation, which gives
$$2k = r + \tfrac{8a}{A-6}.$$
Ponder this a moment. Since $\tfrac{r}{2}$ is odd, we know that $\tfrac{r^2}{4} = \tfrac{A}{4}k + 1 \equiv 1 \mod 4$, and so $A \equiv 0 \mod 16$. If $A \geq 16$, then it must follow that $2k \leq r + \tfrac{4}{5}a < 2k$, which is nonsense. If $A = 0$, then we find instead $2k = r -\tfrac{4}{3}a < 2k$, also a contradiction.
\item[D3, F5] Write
$$f_{D,3}(r) = (r^4-8r^2)+ (8k^2r-16k^2+16).$$
The two bracketed expressions are both positive once $r \geq 4$, but since $r \equiv 0 \mod 4$, it follows that $r = 4$ is the smallest value for $r$ allowed. Hence we have our contradiction.
\item[E3, F3]From condition 3 we know that $B = \tfrac{r}{2}-\tfrac{a}{2} - \tfrac{A}{2} \leq 1-\tfrac{A}{2} < 0$ unless $A= 0$. This contradicts conditions E and F. However, if $A = 0$, then $r = 2$  and we violate the condition that $r \equiv 0 \mod 4$.
\item[D4]Write
$$f_{D,4}(r) = \underbrace{(r^4-4kr^3-k^2r^2+4k^3r-8k^3)}_{g(r)} + (8k^2r-2kr^2) + (48k^2-8r^2)+16kr +8k +16$$
and observe that except possibly $g(r)$, all the terms are positive. We aim to show that on the range $2 < r < k$ we have $g(r)>0$. Indeed, consider its derivatives:
\begin{align*}
\frac{dg}{dr}(r) &= 4r^3-12kr^2-2k^2r+4k^3\\
\frac{d^2g}{dr^2}(r) &= 12r^2-24kr-2k^2.
\end{align*}
Now, the second derivative is clearly negative on $0 < r < k$, and thus on our range of interest $\tfrac{dg}{dr}$ is decreasing. Observing that $\tfrac{dg}{dr}(0)=4k^3>0$ and $\tfrac{dg}{dr}(k)=-6k^3<0$ we know there is precisely one zero to $\tfrac{dg}{dr}$ on $0 < r < k$. That is, $g$ has precisely one turning point, and since $\tfrac{d^2g}{dr^2} < 0$ it is a local maximum. We compute:
\begin{align*}
g(4) = 8k^3-16k^2-256k+256 & & g(k-2)=4k^3-28k^2+16.
\end{align*}
When $k \geq 7$, these are both positive, so the function is positive over the range $4 \leq r \leq k-2$. Notice that the requirements that $r\equiv 0 \mod 4$ and $r^2 \equiv 4 \mod k$ both imply that we need not consider $r = 2,k-1$, and so this suffices for our contradiction. If $k = 5$, Proposition \ref{primep} tells us $r=2$, $A = 0$, $a= 4$, and cannot be in this case since condition 4 is violated.
\item[E4]Rearrange the $Z(r)-Z(0)=0$ equation to obtain
$$4k = r - \tfrac{4}{A-2}(r-2a).$$
At this point we know (from condition 4) that $a\leq r-2$, whence $4k \leq r + \tfrac{4k}{A-2} < 3k$ if $A \neq 0$, since $A \equiv 0 \mod 4$. If $A = 0$, then $a = \tfrac{k+3}{2} > 0=r-2$, a contradiction.
\item[F4] Write
\begin{align*}
f_{F,4}(r) &= (r^4  +k^2r^2 -2k^3r  ) - 4kr^3 -( 2k + 8)r^2 \\
&\phantom{SPACE}- (2k^3 - 24k^2 - 16k)r - (72k^3 - 48k^2 - 8k - 16).
\end{align*}
Once $k\geq 13$, all the bracketed terms are negative. For $k<13$, we contradict conditions F and 4 by way of Proposition \ref{primep} since $k$ must be prime power.
\item[D5] Write
$$f_{D,5}(r) = (r^4 - 2k^2r^2 - 8r^2) - 8k^2r - (16k^2 - 16)$$
and note that all bracketed terms are negative.
\item[E5] Write
$$f_{E,5}(r) = g(r) + (2k-8)r^2 + (8k^3 -8k^2r+16kr) + (8k^3 -16k^2 - 8k + 16),$$
where $g(r)$ is as in case D4, and all bracketed terms are positive if $k \geq 7$. If $k = 5$, we are not in this case by Proposition \ref{primep}.
\end{description}

With all possibilities checked, we are finished the proof.
\end{proof}

\subsubsection{The Case $a$ Odd}

We now repeat for $a$ an odd integer. This is extremely similar to the previous situation, so we omit proofs which are virtually identical. As with Step One before, the proofs of the following are straightforward verifications.

\begin{prop}[$r \equiv 2 \mod 4$]
If $r \equiv 2 \mod 4$ and $a$ is odd, then we define parameters $B:= 1 + \tfrac{r}{2}-\tfrac{a-k}{2}-\tfrac{A}{2} \in [0,k-1)$ and $J:= -\tfrac{r}{2}{-4} < 0$, giving
$$\phi(2s\ell-r\ell) =\begin{cases} -K^1_{B,-J-2B} &\text{if }J+2B < k\\ -K^1_{B+2,-J-2B+2k-4} &\text{if }J+2B\geq k\end{cases},$$
and also
$$\phi(2s\ell) = \begin{cases}K^1_{\tfrac{a-r+k}{2},\tfrac{r}{2}-a-k} &\text{if }r > 2a\\
K^1_{\tfrac{a-r+k}{2}+2,\tfrac{r}{2}-a+k-4} &\text{if }r \leq 2a
\end{cases}.$$
\end{prop}

\begin{prop}[$r \equiv 0 \mod 4$]
If, on the other hand, $r \equiv 0 \mod 4$, then we instead define $B := \tfrac{r}{2}-\tfrac{a-k}{2} - \tfrac{A}{2} \in [0,k-1)$ and $J:=k-\tfrac{r}{2}-4>0$, giving
$$\phi(2s\ell-r\ell) = \begin{cases}
-K^1_{B,-J-2B} &\text{if }J+2B < k\\
-K^1_{B+2,-J-2B+2k-4} &\text{if }J+2B \geq k
\end{cases},$$
and lastly
$$\phi(2s\ell) = K^1_{\tfrac{a-r+k}{2}+1,\tfrac{r}{2}-a-2}.$$
\end{prop}

With these computed, we then establish the tables exactly as before. First, $i = r$:

\begin{center}
\begin{tabular}{c|c|c||c}
Case &$r \mod 4$ &Conditions &$16k^2Z(r)$\\
\hline
\hline
$A$ &$2$ & $J+2B < k$ &$(4kr-8k^2)A+(2-3k)r^2+((4a-24)k+4k^2)r$\\
& & & $\phantom{SPACE}+4k^3+(-8a+16)k^2+32ak-8$\\
\hline
$B$ &$2$ & $J+2B \geq k$ &$(4kr+8k^2)A+(2-3k)r^2+((4a-24)k-4k^2)r$\\
& & & $\phantom{SPACE}+4k^3+(8a+48)k^2 +32ak-8$\\
\hline
\hline
$C$ &$0$ & $J+2B<k$ &$(4kr-16k^2)A+(2-3k)r^2+((4a-24)k+8k^2)r$\\
& & & $\phantom{SPACE}+16k^3+(-16a-24)k^2+32ak-8$\\
\hline
$D$ &$0$ & $J+2B \geq k$ &$4Akr+(2-3k)r^2+(4a-24)kr +40k^2+32ak-8$
\end{tabular}
\end{center}

And now for $i = 0$:

\begin{center}
\begin{tabular}{c|c|c||c}
Case &$r \mod 4$ &Conditions &$16k^2Z(0)$\\
\hline
\hline
$1$ &$2$ &$r>2a$ &$(2-3k)r^2 + (4ak - 4k^2)r +4k^3 +8ak^2 -8$\\
$2$ &$2$ &$r \leq2a$ &$(2-3k)r^2 + (4ak + 4k^2)r +4k^3-8ak^2 -8$\\
\hline
$3$ &$0$ &$-$ &$(2-3k)r^2+4akr + 8k^2-8$
\end{tabular}
\end{center}

As before, we have the following proposition.

\begin{prop}
If $a$ is odd, then no $\ell$ exists which ensures that $Z(r) = Z(0) = 0$.
\end{prop}
\begin{proof}
Exactly as before, we have another table (though it is much smaller this time):

\begin{center}
\begin{tabular}{c|c}
$\alpha,\beta$ &$f_{\alpha,\beta}(r)$\\
\hline
\hline
$A,1$ &$r^4-(5k^2-2k+8)r^2+4k^4+16k^2-8k+16$\\
$B,1$ &$r^4 +4kr^3 + (4k^2-8)r^2 + (8k^2-16k)r + 16k^3-16k^2 + 16$\\
\hline
$A,2$ &$r^4-4kr^3+(4k^2-8)r^2+(8k^2+16k)r-16k^3-16k^2+16$\\
$B,2$ &$r^4 - (3k^2 + 2k + 8)r^2 + 16k^2r - 4k^4 - 32k^3 + 16k^2 + 8k + 16$\\
\hline
\hline
$C,3$ &$r^4-4kr^3-(k^2-2k+8)r^2 +(4k^3-8k^2+16k)r +8k^3-16k^2-8k+16$\\
\hline
$D,3$ &$r^4 -8r^2 + 8k^2r -16k^2+16$\\
\end{tabular}
\end{center}

The case-by-case analysis goes as follows.
\begin{description}
\item[A1] We observe that $f_{A,1}$ has the same structure as cases A2 and B1 from the previous section, and that $f_{A,1}(k-1) = 8k^3+5k^2+6k+9 > 0$, whence we are done.
\item[B1] Each of the coefficients of $r$ in $f_{B,1}(r)$ is clearly positive.
\item[A2] The $Z(r)-Z(0) =0$ equation gives us
$$k=\tfrac{4a-2r}{A-2} + \tfrac{r}{2},$$
and since $A\equiv 0 \mod 16$ (see case C2 in the previous section), we discover, barring $A = 0$, that $k<\tfrac{1}{7}(2a-r) + \tfrac{r}{2} < \tfrac{11}{14}k$, which is a contradiction. If $A=0$, we see $k = r-2a + \tfrac{r}{2} < k$ (since $r \leq 2a$ by condition 2), also a contradiction.
\item[B2] From condition B, we see $J+2B \geq k$, so $\tfrac{r}{2} -a -A-2 \geq 0$. However, from condition 2, we know that $r \leq 2a$, so we have a contradiction.
\item[C3] We write
$$f_{C,3}(r)=g(r) +(2k-8)r^3 +(8k^3-8k^2r+16kr) +(8k^3-16k^2-8k+16),$$
where $g(r)$ is the same function as in case D4 above. We know that all terms are positive for $k\geq 7$, and if $k = 5$ we obtain the usual contradiction (namely, we are not in this case).
\item[D3] As cases D3 and F5 from the previous section.
\end{description}
All cases are done, and so is the proof.
\end{proof}

\subsubsection{The Proof of Theorem \ref{0}}

\begin{proof}[Proof of Theorem \ref{0}]
For any $m$, we know $k = 1$ yields the unknot. Otherwise, we know by Lemma \ref{Alexander} that $m = 1$. Moreover, we now know from the previous two subsections that if $k\geq 5$ then $P(k,-k,2)$ cannot have unknotting number one. Since $P(3,-3,2)$ does indeed have unknotting number one, the theorem is proved.
\end{proof}

\subsection{Examples}

To illustrate the above working, we focus on the case that $k$ is prime power. Recall from Proposition \ref{primep}, there is an essentially unique $\ell$. Then $a$ is even or odd according to the congruence of $k$ modulo 4 (cases A1 and A2 respectively). We get:
\begin{align*}
\phi(2\ell) = K^3_{k-4} &&\phi((2s-2)\ell) = \begin{cases}-K^3_k &k=5\\K^1_{\tfrac{1}{4}(k-5),-\half(k+5)} &k>5 \text{ and } k\equiv 1 \mod 4\\-K^1_{\tfrac{1}{4}(k+5),-\half(k-5)} &k\equiv 3 \mod 4 \end{cases}.
\end{align*}
We then find (surprisingly independently of the conditions on $k$ modulo 4):
\begin{align*}
d(\Sigma,\phi(2\ell))= -\tfrac{1}{k^2}(-2k^2+8) &&d(\Sigma,\phi((2s-2)\ell)) = -\tfrac{1}{2k^2}(-k^2+25).
\end{align*}
Grinding all this into \eqref{obs2}, we should find $Z(2) = 0$, but in fact:
$$Z(2) = \tfrac{1}{2k^2}(3k^2 +9) + \tfrac{1}{2k^2}(k^2-9) = 2,$$
which is blatantly untrue.

We can see this even more concretely in a particular example, namely $k = 5$. The correction terms for the lens space in this case are:
$$d(L,i) =(0,-\tfrac{2}{25},-\tfrac{8}{25},-\tfrac{18}{25},-\tfrac{32}{25},-2,-\tfrac{72}{25},-\tfrac{48}{25},-\tfrac{28}{25},-\tfrac{12}{25},0,\tfrac{8}{25},\tfrac{12}{25}, \dots).$$
Here, we have only presented the first half since $d(\cdot,i) = d(\cdot,-i)$. Then for the double cover, we have, using our isomorphism $\phi$,
$$d(\Sigma,i^\prime) =(0,\tfrac{22}{25},-\tfrac{12}{25},-\tfrac{2}{25},\tfrac{2}{25},0,\tfrac{42}{25},\tfrac{28}{25},\tfrac{8}{25},-\tfrac{18}{25},0,\tfrac{12}{25},\tfrac{18}{25}, \dots).$$
Now, solving \eqref{cong} tells us $\ell = \pm 3$, so take $\ell = 22$ and note that indeed $r = 2$, $A = 0$, and $a = 4$. We find
$$d(\Sigma,22i^\prime) =(0,-\tfrac{2}{25},\tfrac{42}{25},-\tfrac{18}{25},\tfrac{18}{25},0,\tfrac{28}{25},\tfrac{2}{25},\tfrac{22}{25},-\tfrac{12}{25},0,\tfrac{8}{25},\tfrac{12}{25}, \dots),$$
and tabulate the corresponding sides of \eqref{obs2}, multiplying them by $25$:
\begin{center}
\begin{tabular}{c||cc}
$i$ &$\Sigma(k,-k,2)$ &$-L(k^2,2)$\\
\hline
$0$ &$-12$ &$-12$\\
$1$ &$-10$ &$-10$\\
$2$ &$42$ &$-8$\\
$3$ &$-6$ &$-6$\\
$4$ &$-4$ &$-4$\\
$5$ &$-2$ &$-2$\\
$6$ &$0$ &$0$
\end{tabular}
\end{center}
We can see here that the two sides are congruent modulo 25, but not equal, so the knot $P(5,-5,2)$ cannot have unknotting number one. We can also explicitly see the failure of $Z(2) = 0$, and that the correct value is indeed $Z(2) = 2$.

For those who wish to compare this with Theorem 1.1 of \cite{OSUnknot}, we remark that our choice of $\ell$ also gives us a positive, even matching. This matching, however, is \emph{not} symmetric.

\section{Further Remarks}
\label{s:2,4}

There are two remaining cases: $P(k,-k+2,2m)$ and $P(k,-k+4,2m)$. In these cases we run into difficulties applying the above methods. First of all, we cannot employ Theorem \ref{greene}, since $\Sigma(K)$ is not an $L$-space. Though some work has or is being done to remove the $L$-space restriction, any application of the theorem to $P(k,-k+2,2m)$ when the signature vanishes can at best isolate the sign of the crossing change, as happened with $P(k,-k,2)$, and even naive use of the obstruction without consideration for these orientation hypotheses fails to resolve these cases fully. We end up with two infinite families in each ($\det K = 1,5$ and $\det K = 3, 11$ respectively) for which the theorem provides no obstruction.

Since the proof of Theoerm \ref{greene} does not require the full power of Theorem \ref{OSunknot} (including the case when $D \equiv 3 \mod 4$ not stated here), one might consider using Theorem \ref{OSunknot} directly. However, even here we have no hope of a complete proof since \eqref{obs} is vacuously satisfied when $\det K = 1, 3$: we do not have enough $\Spinc$-structures for any asymmetries to occur. Therefore, a totally new method will have to be brought to bear if we are to crack these cases.

In both cases, the authors have tried using the Alexander module to no effect, and their computations in small cases suggest that neither the Rasmussen nor Ozsv\'ath-Szab\'o $\tau$-invariant are of any use either. They therefore leave treatment of these two cases to another paper at another time.

\bibliographystyle{plain}
\bibliography{pretbib}
\end{document}